\documentclass[10pt]{article}
\RequirePackage{mathtools}
\usepackage{amssymb}
\usepackage{amsthm}
\usepackage[OT1]{fontenc}
\usepackage{subcaption} 
\usepackage{caption}
\usepackage{multirow}

\usepackage{graphicx}
\usepackage{fullpage}
\usepackage{bm}
\usepackage[linesnumbered,ruled,vlined]{algorithm2e}
\usepackage[colorlinks=true, linkcolor=blue, citecolor=blue]{hyperref}
\usepackage{cleveref}
\usepackage{natbib}
\usepackage{def}
\usepackage{bbm}

\bibpunct{(}{)}{;}{a}{,}{,}

\newtheorem{condition}[theorem]{Condition}

\theoremstyle{remark}

\def\Tcal{{\mathcal T}}

\def\Fcal{{\mathcal{F}}}

\def\Ccal{{\mathcal C}}

\def\Hcal{{\mathcal H}}

\def\Mcal{{\mathcal M}}
\def\Qcal{{\mathcal Q}}

\def\Ocal{{\mathcal O}}

\def\Ebb{{\mathbb E}}

\def\argmin{{\operatorname{argmin}}}

\def\romanone{{\expandafter{\romannumeral1}}}
\def\romantwo{{\expandafter{\romannumeral2}}}
\def\romanthree{{\expandafter{\romannumeral3}}}
\def\TV{{\operatorname{TV}}}

\def\loc{{\operatorname{loc}}}
\def\Bin{{\mathrm{Bin}}}
\def\d{{\operatorname{d}}}
\def\B{{\operatorname{B}}}
\usepackage[affil-it]{authblk}

\title{Robust density estimation over star-shaped density classes}
\author{Xiaolong Liu and Matey Neykov\\
Department of Statistics and Data Science, Northwestern University\\ \texttt{xiaolongliu2025@u.northwestern.edu},~~~ \texttt{mneykov@northwestern.edu}}
\date{}

\begin{document}

\maketitle
\begin{abstract}
We establish a novel criterion for comparing the performance of two densities, \(g_1\) and \(g_2\), within the context of corrupted data. Utilizing this criterion, we propose an algorithm to construct a density estimator within a star-shaped density class, \(\mathcal{F}\), under conditions of data corruption. We proceed to derive the minimax upper and lower bounds for density estimation across this star-shaped density class, characterized by densities that are uniformly bounded above and below (in the sup norm), in the presence of adversarially corrupted data. Specifically, we assume that a fraction \(\epsilon \leq \frac{1}{3}\) of the \(N\) observations are arbitrarily corrupted. We obtain the minimax upper bound \(\max\{ \tau_{\Bar{J}}^2, \epsilon \} \wedge d^2\).
Under certain conditions, we obtain the minimax risk, up to proportionality constants, under the squared \(L_2\) loss as
\[
\max\left\{ \tau^{*2} \wedge d^2, \epsilon \wedge d^2 \right\},
\]
where \(\tau^* := \sup\left\{ \tau : N\tau^2 \leq \log \mathcal{M}_{\mathcal{F}}^{\text{loc}}(\tau, c) \right\}\) for a sufficiently large constant \(c\).
Here, \(\mathcal{M}_{\mathcal{F}}^{\text{loc}}(\tau, c)\) denotes the local entropy of the set \(\mathcal{F}\), and \(d\) is the \(L_2\) diameter of \(\mathcal{F}\).
% To the best of our knowledge, this is the first result establishing the robust minimax rate for density estimation.
%If we have
%\[
%\xi(\epsilon) := \max_{f_1, f_2 \in \mathcal{F}, \, \TV(P_1, P_2) \leq \frac{\epsilon'}{1 - \epsilon'}} \{\|f_1 - f_2\|_2^2\} \gtrsim \epsilon \text{ when } \epsilon \geq \frac{k}{N},
%\]
%where $k > 7$ is a constant, $\epsilon' = \epsilon - \frac{1}{N}$, \(f_1\) and \(f_2\) denote the densities of distributions \(P_1\) and \(P_2\), respectively, \(\TV\) is the total variation, and \(\left\|\cdot\right\|_2\) is the \(L_2\) norm, 
\end{abstract}

\section{Introduction}
When dealing with minimax rates for density estimation, global metric entropy is often employed. Specifically, the following equation, sometimes informally referred to as the `Le Cam equation', is used to heuristically determine the minimax rate of convergence:
$$\log \Mcal_{\mathcal{F}}^{\operatorname{glo}}(\tau) \asymp N\tau^2,$$
where $N$ is the sample size, $\log \Mcal_{\mathcal{F}}^{\operatorname{glo}}(\tau)$ is the global metric entropy of the density set $\mathcal{F}$ at a Hellinger distance $\tau$, and $\tau^2$ determines the order of the minimax rate.

Recently, \cite{shrotriya2023revisiting} established that local metric entropy consistently determines the minimax rate for star-shaped density classes, where the densities are assumed to be (uniformly) bounded from above and below. Specifically, they revised the Le Cam equation by replacing the global entropy with local entropy (see Definition \ref{def local metric}) and the Hellinger metric with the $L_2$ metric. Moreover, they proved that the constraints on the density class could be relaxed to a star-shaped density class that is uniformly bounded above and contains a density that is bounded below. They proposed a `multistage sieve' maximum likelihood estimator that achieves these bounds.

In this paper, we extend this scenario to adversarial data. Specifically, while still seeking an estimator for the true density within this bounded star-shaped density class, we now face corrupted data. We assume that a fraction $\epsilon \leq \frac{1}{3}$ of the $N$ observations are arbitrarily corrupted. We modify the `multistage sieve' algorithm from \cite{shrotriya2023revisiting} accordingly. Specifically, we establish a new criterion to compare the performance of any two densities $g_1$ and $g_2$ based on the corrupted data (see \eqref{group log likelihood difference}). By using the local entropy of the set $\mathcal{F}$, we obtain the minimax upper and lower bounds up to proportionality constants under the squared $L_2$ loss. These two bounds do not always match, but under Condition \ref{Condition: L_2 and TV}, we are able to determine the minimax risk, up to proportionality constants.

%To the best of our knowledge, this is the first result on the robust minimax rate of density estimation.

We define 
$$\xi(\epsilon) := \max_{f_1, f_2 \in \mathcal{F}, \, \TV(P_1, P_2) \leq \frac{\epsilon'}{1 - \epsilon'}} \left\{\|f_1 - f_2\|_2^2\right\},$$
where \(f_1\) and \(f_2\) denote the densities of distributions \(P_1\) and \(P_2\), respectively, \(\TV\) is the total variation distance, \(\left\|\cdot\right\|_2\) is the \(L_2\) norm, and \(\epsilon' = \epsilon - \frac{1}{N}\).

\begin{condition}[Condition for optimality]\label{Condition: L_2 and TV}
Assume we have
$
\xi(\epsilon) \gtrsim \epsilon \text{ when } \epsilon \geq \frac{k}{N},
$
where \(k > 7\) is a constant. For $\epsilon < \frac{k}{N}$, no further assumptions are made.
\end{condition}

\begin{remark}
    As demonstrated in the proof of Lemma \ref{new lower bound}, we established that $\xi(\epsilon) \lesssim \epsilon$. Therefore, under Condition \ref{Condition: L_2 and TV}, it follows that $\xi(\epsilon) \asymp \epsilon$ when $ \epsilon \geq \frac{k}{N}$. Furthermore, Theorem \ref{Minimax Rate} shows that, under this condition, our algorithm achieves optimality.
\end{remark}

\begin{remark}
The following provides a simple example where Condition \ref{Condition: L_2 and TV} holds. 
Consider the case where $\mathcal{F}$ includes the constant function \( f(x) := 1 \) for \( x \in [0, 1] \), and let the function \( g(x) \) be defined as:
\[
g(x) := 
\begin{cases} 
\frac{1}{2}, & 0 \leq x < \frac{\epsilon}{2}, \\
1, & \frac{\epsilon}{2} \leq x \leq 1 - \frac{\epsilon}{2}, \\
\frac{3}{2}, & 1 - \frac{\epsilon}{2} < x \leq 1,
\end{cases}
\]
where \( \epsilon \) is a small positive constant. We can verify that the total variation distance satisfies \( \TV(f,g) \asymp \epsilon \) and the \( L_2 \) norm satisfies \( \|f - g\|_2 \asymp \sqrt{\epsilon} \). Therefore, \( \xi(\epsilon) \gtrsim \epsilon \), where \( \xi(\epsilon) \) is defined in Lemma \ref{new lower bound}. 

In this case, by Theorem \ref{Minimax Rate}, the proposed estimator achieves the minimax rate. For instance, the class \( \mathcal{F} \) can be taken as the set of monotone densities that includes both \( f \) and \( g \).
\end{remark}

We adopt the same definition for a general class of bounded densities as in \cite{shrotriya2023revisiting}, denoted by $\mathcal{F}_B^{[\alpha,\beta]}$. We assume that the true density of interest lies within a known star-shaped subset of this ambient density class. Furthermore, we will adhere to the definitions provided in \cite{shrotriya2023revisiting}, specifically Definitions \ref{def:KL-divergence}, \ref{def:packing sets}, and \ref{def local metric}.

\begin{definition}[Ambient density class $\Fcal_B^{[\alpha,\beta]}$]\label{def of F}
    Given constants $0 < \alpha < \beta < \infty$, for some fixed dimension $p \in \mathbb{N}$, and a common known (Borel measurable) compact support set $B \subseteq \mathbb{R}^p$ (with positive measure), we then define the class of density functions, $\Fcal_B^{[\alpha,\beta]}$, as follows:

\[
\Fcal_B^{[\alpha,\beta]} := \left\{ f : B \to [\alpha, \beta] \ \middle| \ \int_B f \ d\mu = 1, f \text{ measurable} \right\}
\]
where $\mu$ is the dominating finite measure on $B$. We always take $\mu$ to be a (normalized) probability measure on $B$.

Furthermore, we endow $\Fcal_B^{[\alpha,\beta]}$ with the $L_2$-metric. That is, for any two densities $f, g \in \Fcal_B^{[\alpha,\beta]}$, we denote the $L_2$-metric between them to be

\[
\|f - g\|_2 := \left( \int_B (f - g)^2 \ d\mu \right)^{\frac{1}{2}}. 
\]
\end{definition}

\begin{remark}
    When $\epsilon$ is sufficiently small, specifically less than $d_1/3$, where $d_1$ is the diameter of $\Fcal$ in the total variation distance, it becomes feasible to identify two densities whose distance is exactly $\epsilon$, see Lemma $1.3$ in \cite{AkshayPrasadan2024}. Note that by using Holder's inequality, we have $\int |f_1 - f_2|d\mu \leq (\int (f_1 - f_2)^2d\mu)^{1/2}$. Therefore, we have $\TV(f_1, f_2) \lesssim \|f_1 - f_2\|_2$. Consequently, we have $\xi(\epsilon) \gtrsim \epsilon^2$. Therefore, for sufficiently small $\epsilon$, the relationship $\epsilon \gtrsim \xi(\epsilon) \gtrsim \epsilon^2$ holds.
\end{remark}

\begin{remark}
    As mentioned and analyzed in \cite{shrotriya2023revisiting}, Definition \ref{def of F} implies that $\mathcal{F}_B^{[\alpha,\beta]}$ forms a convex set, and that the metric space $\left( \mathcal{F}_B^{[\alpha,\beta]}, \left\|\cdot\right\|_2 \right)$ is complete, bounded, but may not be totally bounded.
\end{remark}

\textit{\textbf{Core Problem:}} Suppose we observe \(n\) observations \(\mathbf{\tilde{X}} := (\tilde{X_1}, \dots, \tilde{X_n})^\top \overset{\text{i.i.d.}}{\sim} f\), where \(f \in \Fcal\) is a fixed but unknown density function. Here, \(\Fcal \subset \Fcal_B^{[\alpha, \beta]}\) is a known star-shaped set. For a fixed \(\epsilon \in [0, \frac{1}{3}]\), independent of \(n\), a fraction \(\epsilon\) of the observations in \(\mathbf{\tilde{X}}\) are arbitrarily corrupted by some procedure \(\mathcal{C}\). Let \(\mathbf{X} = \mathcal{C}(\mathbf{\tilde{X}})\), where the \(i\)th coordinate is \(X_i = \mathcal{C}(\tilde{X}_i)\) if it is corrupted, otherwise \(X_i = \tilde{X}_i\). Our goal is to propose a universal estimator for \(f\) based on this corrupted data and to derive the exact (up to constants) squared \(L_2\)-minimax rate of estimation in expectation.

If Condition \ref{Condition: L_2 and TV} holds, our estimator \(\nu^*\) achieves the following minimax rate:
\[
\inf_{\hat{\nu}} \sup_{f} \mathbb{E}_f\|\hat{\nu} - f\|_2^2 \asymp \max\{{\tau^*}^2 \wedge d^2, \epsilon \wedge d^2\}.
\]
Here \(\tau^* := \sup\{\tau : N\tau^2 \leq \log \Mcal^{\loc}_{\Fcal}(\tau, c)\}\), with \(\log \Mcal^{\loc}_{\Fcal}(\tau, c)\) being the \(L_2\)-local metric entropy of \(\Fcal\) (see Definition \ref{def local metric}). The quantity \(d := \operatorname{diam}_2(\Fcal)\) refers to the \(L_2\)-diameter of \(\Fcal\), which is finite due to the boundedness of \(\Fcal_B^{[\alpha,\beta]}\) in our setting.

\begin{remark}[Extending results to \(\mathcal{F}_B^{[0,\beta]}\)]
    If the class \(\mathcal{F}\) contains at least one density function that is bounded away from \(0\), then the boundedness requirement on \(\mathcal{F}\) can be relaxed to \(\mathcal{F} \subset \mathcal{F}_B^{[0,\beta]}\), where \(\mathcal{F}\) is a star-shaped set. This result is proved in \cite{shrotriya2023revisiting} and is shown in Proposition \ref{Extending results to zero lower bound}.
\end{remark}

\subsection{Related Literature}
There is an extensive body of research focused on deriving minimax lower bounds in density estimation. For instance, \cite{boyd1978lower} established a fundamental lower bound for the mean integrated \( p- \)th power error, showing that \(\mathbb{E}_f\int_{-\infty}^\infty |f(x) - \hat{f}_n(x)|^p \, \d x \gtrsim n^{-1}\) (with \( p \geq 1 \)) for any arbitrary density estimator. \cite{devroye1983arbitrarily} further explored lower bounds for the normalized error \(\mathbb{E}_f\int_{-\infty}^\infty |f(x) - \hat{f}_n(x)|^p \, \d x / \int_{-\infty}^\infty f^p(x) \, \d x\) in density estimation. 

\citet{bretagnolle1979estimation} assumed that the density function \( f \) belongs to a subset \( \mathcal{F} \) characterized by a certain level of smoothness. They derived sharp lower bounds for the minimax risk under the \( L_p \) norm within this smoothness-constrained class. Since a smoothness assumption for a family of densities generally implies certain dimensional properties, as discussed by \citet{birge1986estimating}, they assumed specific dimensional properties for the density families and used metric entropy-based methods to derive sharp risk bounds for many smooth density families under the Hellinger metric. \cite{efroimovich1982estimation} considered the problem where the density is assumed to belong to an ellipsoid in a Hilbert space and is square-integrable. They used an orthonormal basis \(\{u_j(x)\}_{j = 0, 1, \dots}\) to uniquely represent the true density as \(f(x) = \sum_{j = 0}^\infty \theta_j u_j(x)\), where \(\theta_j = \int f(x)u_j(x)\d \nu\). This approach transforms the problem into the estimation of the infinite-dimensional parameter \(\theta\). The parameter \(\theta\) is assumed to belong to an ellipsoid defined as \(\Theta := \{\theta: \sum_{j=0}^\infty a_j \theta_j^2 \leq 1 \text{ and } \lim_{j \rightarrow \infty}a_j = \infty\}\). They then derived both the asymptotic lower bounds for the minimax risk and the corresponding upper bounds. \cite{has1978minskii} utilized Fano's lemma to derive lower bounds for density estimation in the \( L_2 \) metric under specific smoothness conditions. \cite{Ibragimov1997} applied Fano's lemma to establish minimax lower bounds for estimating a shifted parameter in infinite-dimensional Gaussian white noise. 

\cite{ibragimov1978capacity} explored the capacity of communication channels defined by the stochastic differential equation \( \d X(t) = S(t) \, \d t + \epsilon \, \d w(t) \), where \( w(t) \) is a standard Wiener process and \( \epsilon \) is a small positive parameter. They focused on the problem of estimating an unknown function \( S \), given observations of the process \( X(t) \) over the interval \([0, 1]\). The function \( S \) was assumed to belong to a predefined subset \( \mathcal{S} \) of the unit ball in the \( L^2(0, 1) \) space. By leveraging an argument analogous to Fano's inequality, the authors derived a lower bound on the minimax risk for the estimation problem under the \( L_p \) norm. In \cite{Ibragimov1983}, the authors considered the problem of estimating a density function \( f \) from a known class of functions \( \mathcal{F} \). They utilized Fano's lemma to derive a minimax lower bound for the estimation error measured in the \( L_p \) norm, where \( 2 \leq p \leq \infty \). The derivation is based on the assumption that within the class \( \mathcal{F} \), there exist \( p(\delta) \) distinct densities \( f_{i\delta} \) such that the distance between any two distinct densities satisfies \( \|f_{i\delta} - f_{j\delta}\| \geq \delta \), and the loss function \( L \) satisfies \( L(\delta/2) > 0 \). Here, \( p(\delta) \) is an integer that depends on \( \delta \), which is a positive constant. \cite{10.1214/aos/1176347736} employed Fano's lemma to derive minimax lower bounds under the scenario where the unknown density \( f \) lies within a known set \(\mathcal{F}\). They argued that if \(\phi \in \Phi \subset \mathcal{F}\) is used to estimate \( f \), and if \(\mathcal{F}\) consists of finite sets with a fixed number \( N \) of elements, the results are closely tied to entropy. \cite{Yu1997} presented several lemmas for deriving optimal lower bounds using both Assouad's and Fano's lemma arguments for densities on compact supports. \cite{yang1999information} demonstrated that global metric entropy bounds effectively capture the minimax risk over an entire function class or substantial subsets of it.

The minimax lower bound techniques, particularly in the context of nonparametric density estimation, are well-documented in classical references such as \cite{Devroye1987}, \cite{Devroye1985}, and \cite{LeCam1986}. More recent works, including \cite{Tsybakov2009} and \cite{wainwright2019high}, also offer valuable insights.

In addition to the extensive research on minimax lower bounds, there has also been significant work on deriving minimax upper bounds in density estimation. For example, \cite{86996} employs the minimum distance principle to derive density estimators defined by \(\hat{f} := \argmin_{q \in \Gamma} \left(L(q) + \log\frac{1}{\prod_{i=1}^n q(X_i)}\right)\), where \(\Gamma\) is a set of candidate densities and \(L(q)\) is a nonnegative function that satisfies Kraft's inequality \(\sum_q 2^{-L(q)} \leq 1\). They use relative entropy (Kullback-Leibler divergence) to upper bound the error under the Hellinger distance. Similarly, \cite{YannisG.Yatracos1985} explores minimum distance estimators, providing uniformly consistent robust estimators and using entropy to derive upper bounds under the \(L_1\) metric. \cite{birge1983approximation} investigates the speed of estimation under the Hellinger metric. Specifically, they consider the minimax risk defined as 
\[
R_n(q) := \inf_{T_n} \sup_{\theta \in \Theta} \Ebb_\theta[d^q(\theta, T_n)],
\]
where \(T_n\) is any estimator of the parameter \(\theta\). The speed of estimation is described by a function \(r(n)\) such that there exist constants \(C_1\) and \(C_2\) satisfying 
\[
C_1r^q(n) \leq R_n(q) \leq C_2r^q(n).
\]
As mentioned earlier, \cite{birge1986estimating} leverages the dimensional properties of density families and uses metric entropy to derive sharp risk bounds under the Hellinger metric. \cite{van1993hellinger} examines the maximum likelihood estimator \(\hat{\theta}_n\) of the true parameter \(\theta\), based on the empirical distribution \(P_n\). The estimator \(\hat{\theta}_n\) is defined as the maximizer of the log-likelihood function, given by 
\[
\int \log(f_{\hat{\theta}_n}) \, dP_n = \max_{\theta \in \Theta} \int \log(f_\theta) \, dP_n.
\]
The study investigates the convergence of \(\hat{f} := f_{\hat{\theta}_n}\) to the true density \(f\) under the Hellinger metric.
\cite{wong1995probability} investigates upper bounds for density estimation using the `sieve' MLE technique, utilizing a metric constructed from a family of discrepancy indices, which includes the squared Hellinger distance, Kullback-Leibler divergence, and Pearson's \(\chi^2\) statistic as special cases. A ‘sieve’ estimator, as noted in \cite[Chapter 8]{Grenander1981}, effectively estimates the parameter of interest through an optimization procedure (e.g., maximum likelihood) over a constrained subset of the parameter space. \cite{birge1993rates} examines minimum contrast estimators (M.C.E.) in a nonparametric setting, deriving upper bounds for their performance, with the main theorem linking the rate of convergence of these estimators to the entropy structure of the parameter space. Furthermore, \cite{Birge1998} extends this work by focusing on minimum contrast estimators on sieves, computing sharp rates of convergence. More recently, \cite{shrotriya2023revisiting} introduced local metric entropy to determine the minimax rate for star-shaped density classes where densities are uniformly bounded above and below. They revised the Le Cam equation using local entropy and the \(L_2\) metric, relaxed constraints on the density class, and proposed a `multistage sieve' MLE that achieves these bounds.

\cite{hampel1968contributions}, \cite{huber1992robust}, and \cite{tukey1960survey} are seminal works in robust statistics that have been systematically studied. Huber's contamination model \cite{huber1965robust, huber1992robust} is a well-known framework for studying robustness in various modern settings. In this model, observations are assumed to be i.i.d. from the distribution \( P_{\epsilon, f, G} := (1-\epsilon)P_f + \epsilon G \), where \( P_f \) is a distribution on \(\mathbb{R}^d\) with the density of interest \( f \), \( G \) is an arbitrary contamination distribution, and \(\epsilon\) represents the contamination proportion.
\cite{chen2016general} investigates the robust minimax risk, \(\inf_{\hat{\theta}}\sup_{\theta \in \Theta}\sup_Q\mathbb{E}_{\epsilon, \theta, Q}L(\hat{\theta}, \theta)\), for a given parameter space \(\Theta\) and a loss function \(L\) under Huber's contamination model: \((1-\epsilon)P_\theta + \epsilon Q\). They demonstrated that this robust minimax risk is upper bounded by a universal constant times \(\min_{\delta>0}\left\{\frac{\log \mathcal{M}(\delta, \Theta, \TV)}{n} + \delta^2\right\} \vee \epsilon^2\), where \(\mathcal{M}(\delta, \Theta, \TV)\) denotes the \(\delta\)-covering number of \(\Theta\) with respect to the total variation distance. In our opinion, this is probably the closest paper to ours on the topic since it provides entropic upper bounds and uses a testing lemma to achieve the rate. We would therefore like to underscore the key differences between our work and \cite{chen2016general}:
\begin{enumerate}
    \item They work with the Huber model while we work with an adversarial model;
    \item They use a $\operatorname{TV}$-based loss function while we use the squared $L_2$ loss;
    \item We provide bounds on the expected squared $L_2$ loss while they give high-probability bounds;
    \item In our opinion there is no corresponding general lower bound to their upper bound. Even though \cite{chen2016general} cite \cite{yang1999information} and claim there is a matching lower bound we are unaware how one can derive a matching bound of the sort $\min_{\delta>0}\left\{\frac{\log \mathcal{M}(\delta, \Theta, \TV)}{n} + \delta^2\right\}$;
    \item We provide upper and lower bounds for general density estimation problems. If Condition \ref{Condition: L_2 and TV} holds, then these bounds are matching. The dependence on the amount of data corruption $\epsilon$ is different from the one exhibited in \cite{chen2016general} as they have $\epsilon^2$ (on the squared TV loss function), while our rate is $\epsilon$ on the squared $L_2$ loss function;
    \item Our testing lemma is completely different from the one used by \cite{chen2016general}. We use a likelihood ratio based test while \cite{chen2016general} use a test function called `Scheffe's estimate', which is inspired by \cite{YannisG.Yatracos1985}'s famous estimator, in order to achieve their entropic upper bound.
\end{enumerate}
In \cite{chen2017robustcovariancescattermatrix}, a general minimax lower bound on the probability for the \(\epsilon\)-contamination in Huber's model is provided. Specifically, it is proven that under certain assumptions, for any \(\epsilon \in [0, 1]\), the inequality \(\inf_{\hat{\theta}} \sup_{\theta \in \Theta} \sup_Q \mathbb{P}_{\epsilon,\theta,Q} \{ L(\hat{\theta}, \theta) > \mathcal{M}(\epsilon)\} > c\) holds, where \(\mathcal{M}(\epsilon) \asymp \mathcal{M}(0) \vee \omega(\epsilon, \Theta)\), and \(\omega(\epsilon, \Theta) = \sup \{L(\theta_1, \theta_2): \TV(P_{\theta_1}, P_{\theta_2}) \leq \epsilon/(1-\epsilon); \theta_1, \theta_2 \in \Theta\}\). This result provides inspiration for deriving a lower bound for the minimax rate within our specific problem setting.
\cite{liu2018densityestimationcontaminateddata} investigates density estimation under pointwise loss in Huber's contamination model and derives the minimax rate for the estimator of \(f(x_0)\) at a point \(x_0 \in \mathbb{R}\), assuming H\"older smoothness for both the true density and the contamination density. \cite{gao2020robust} studies robust regression in the settings of Huber’s contamination models. \cite{zhang2023adaptive} adopt a general \( L_p \, (1 \leq p < \infty)\) loss and assume that the function \( f \) belongs to an anisotropic Nikol'skii class, which serves as a natural extension of the H\"older class under the \(L_p\) norm on \(\mathbb{R}^d\). They address the problem of adaptive minimax density estimation on \(\mathbb{R}^d\) with \(L_p\) loss functions under Huber's contamination model. 

In our problem setting, there is a key difference compared to Huber's contamination model: For a given sample of size \(N\), we have a known fraction \(\epsilon\) of the observations coming from a contamination density, with no randomness in the number of contaminated observations. In contrast, under Huber's contamination model, the number of contaminated observations in a sample of size \(n\) includes some randomness.

\subsection{Organization}
The remainder of this paper is organized as follows. In Section \ref{main results}, we establish risk bounds for our setting, starting with the demonstration of the key topological equivalence between the $L_2$-metric and the Kullback-Leibler divergence in $\mathcal{F}_B^{[\alpha,\beta]}$. Next, in Section \ref{Lower bound}, we derive minimax lower bounds for our setting, introducing additional mathematical background as needed, such as local metric entropy. In Section \ref{Upper Bound Section}, we define our likelihood-based estimator and provide the intuition behind its construction. We then present our algorithm for constructing the estimator using a tree structure, followed by a derivation of its minimax risk upper bound. In Section \ref{Minimax Rate Section}, we derive the minimax rate for our algorithm. Finally, in Section \ref{discussion}, we summarize our findings and discuss potential directions for future research.

\subsection{Notation}
We outline some commonly used notation here. We use $a \vee b$ and $a \wedge b$ to denote the maximum and minimum of two numbers $\{a, b\}$, respectively. Throughout the paper, $\left\|\cdot\right\|_2$ denotes the $L_2$-metric in $\mathcal{F}$. For integers \(n\) and \(m\) with \(n \leq m\), \([n, m]\) denotes the set \(\{n, n + 1, \dots, m\}\). If \(n = 1\), we will simply write \([m]\). We use $\B_2(\theta, r)$ to denote a closed $L_2$-ball centered at the point $\theta$ with a positive radius $r$. The symbols $\lesssim$ and $\gtrsim$ are used to indicate $\leq$ and $\geq$ up to absolute (positive) constant factors, respectively. For two sequences $a_n$ and $b_n$, we write $a_n \asymp b_n$ if both $a_n \lesssim b_n$ and $a_n \gtrsim b_n$ hold. Throughout the paper, $\log$ refers to the natural logarithm. Our use of $\{\alpha,\beta\}$ exclusively refers to the constants in Definition \ref{def of F}, defining $\mathcal{F}_B^{[\alpha,\beta]}$ (and thus $\mathcal{F}$).

\section{Main Results}\label{main results}
First, we have the result that there exists a `topological equivalence' between the $L_2$-metric and the KL-divergence on the density class $\mathcal{F}_B^{[\alpha,\beta]}$. This result was stated without proof in \cite{klemela2009smoothing} and later proved in \cite{shrotriya2023revisiting}.
\begin{definition}[KL-divergence]\label{def:KL-divergence}
    For any two densities $f, g \in \Fcal_B^{[\alpha,\beta]}$, the KL-divergence between them is defined to be
\[
d_{\text{KL}}(f \| g) := \int_B f \log \left( \frac{f}{g} \right) \, d\mu = \mathbb{E}_f \log \left( \frac{f(X)}{g(X)} \right),
\]
where $X \sim f$.
\end{definition}

The following is Lemma 2 in \cite{shrotriya2023revisiting}.
\begin{lemma}[KL-$L_2$ equivalence on $\Fcal_B^{[\alpha,\beta]}$]\label{lemma2 in shrotriya}
    For each pair of densities $f,g \in \Fcal_B^{[\alpha,\beta]}$, the following relationship holds:
\begin{equation}
c(\alpha, \beta) \|f - g\|_2^2 \leq d_{KL}(f\|g) \leq \frac{1}{\alpha} \|f - g\|_2^2,
\end{equation}
where we denote $c(\alpha, \beta) := \frac{h(\beta/\alpha)}{\beta} > 0$. Here $h : (0, \infty) \rightarrow \mathbb{R}$ is defined to be
\begin{equation}
h(\gamma) := 
\begin{cases}
\frac{\gamma - 1 - \log \gamma}{(\gamma - 1)^2} & \text{if } \gamma \in (0, \infty) \setminus \{1\} \\
\frac{1}{2} = \lim_{x \rightarrow 1} \frac{x - 1 - \log x}{(x - 1)^2} & \text{if } \gamma = 1,
\end{cases}
\end{equation}
and is positive over its entire support. It is also easily seen that on $\Fcal_B^{[\alpha,\beta]}$, $d_{KL}$ (and hence the $L_2$-metric) is also equivalent to the Hellinger metric. Furthermore, these properties are also inherited by $\Fcal \subset \Fcal_B^{[\alpha,\beta]}$, which is our density class of interest.
\end{lemma}

\subsection{Lower Bounds}\label{Lower bound}
We first establish lower bounds for the density estimator. For completeness, we will introduce Fano's inequality for the star-shaped density class $\mathcal{F}$ [see \cite{Tsybakov2009}, Lemma 2.10].

\begin{lemma}[Fano's inequality for $\Fcal$]\label{Fano's inequality}
    Let $\{ f^1, \ldots, f^m \} \subset \Fcal$ be a collection of $\tau$-separated densities (i.e., $\| f^i - f^j \|_2 > \tau$ for $i \neq j$), in the $L_2$-metric. Suppose $J$ is uniformly distributed over the index set $[1, m]$, and $(X_i \mid J = j) \sim f^j$ for each $i \in [1, N]$. Then
\[
\inf_{\hat{\nu}} \sup_{f} \mathbb{E}_f \| \hat{\nu}(\mathbf{X}) - f \|_2^2 \geq \frac{\tau^2}{4} \left( 1 - \frac{I(\mathbf{X}; J) + \log 2}{\log m} \right),
\]
where $I(\mathbf{X}; J) \leq \sum_{i=1}^N I(X_i; J)$ is the mutual information between $\mathbf{X}$ and $J$. 

Here $I(X_1; J) := \frac{1}{m}\sum_{j = 1}^md_{KL}(f^j||\Bar{f})$ with $\Bar{f} := \frac{1}{m}\sum_{j = 1}^m f^j$. [See Section $15.3$ in \cite{wainwright2019high}].
\end{lemma}

% \begin{remark}
%     The above inequality $I(\vec{X}; J) \leq \sum_{i=1}^n I(X_i; J)$ holds because 
% \begin{align*}
%     I(\vec{X}; J) &= H(\vec{X}) - H(\vec{X} | J)\\
%     &\leq \big[ \sum_{i=1}^n H(X_i) \big] - H(\vec{X} | J)\\
%     &= \sum_{i=1}^n\{ H(X_i) - H(X_i | J) \}\\
%     &= \sum_{i=1}^n I(X_i; J),
% \end{align*}
% where $H(X|Y)$ is the conditional entropy.
% \end{remark}

Next, we provide the classical definition of packing sets and packing numbers.

\begin{definition}[Packing sets and packing numbers of $\Fcal$ in the $L_2$-metric]\label{def:packing sets}
Given any $\tau > 0$, an $\tau$-packing set of $\Fcal$ in the $L_2$-metric, is a set $\{f^1, \dots, f^m\} \subset \Fcal$ of $\tau$-separated densities (i.e., $\|f^i - f^j\|_2 > \tau$ for $i \neq j$) in the $L_2$-metric. The corresponding $\tau$-packing number, denoted by $\Mcal(\tau, \Fcal)$, is the cardinality of the maximal $\tau$-packing of $\Fcal$. We refer to $\log \Mcal^{\operatorname{glo}}_{\Fcal}(\tau) := \log \Mcal(\tau, \Fcal)$ as the global metric entropy of $\Fcal$.
\end{definition}

Now, we define the notion of local metric entropy, which plays a key role in the development of our risk bound.
\begin{definition}[Local metric entropy of $\Fcal$]\label{def local metric}
 Let $c > 0$ be fixed, and $f \in \Fcal$ be an arbitrary point. Consider the set $\Fcal \cap \B_2(f, \tau)$. Let $\Mcal(\tau/c, \Fcal \cap \B_2(f, \tau))$ denote the $\tau/c$-packing number of $\Fcal \cap \B_2(f, \tau)$ in the $L_2$-metric. Let

\[
\Mcal^{\loc}_{\Fcal}(\tau, c) := \sup_{f \in \Fcal} \Mcal(\tau/c, \Fcal \cap \B_2(f, \tau)) =: \sup_{f \in \Fcal} \Mcal^{\operatorname{glo}}_{\Fcal \cap \B_2(f, \tau)}(\tau/c).
\]

We refer to $\log \Mcal^{\loc}_{\Fcal}(\tau, c)$ as the \textit{local metric entropy of} $\Fcal$.

\end{definition}

The following is Lemma 6 of \cite{shrotriya2023revisiting}. By applying Fano's inequality directly, we obtain a minimax lower bound for our star-shaped density estimation setting over $\mathcal{F}$. Notably, this lower bound does not depend on the corruption rate $\epsilon$. Immediately afterward, we establish a lower bound that accounts for the corruption level $\epsilon$.

\begin{lemma}[Minimax lower bound in general case]\label{lowerbound in shrotriya}
    Let $c > 0$ be fixed, and independent of the data samples $\mathbf{X}$. Then the minimax rate satisfies
$$\inf_{\hat{\nu}}\sup_{f\in\Fcal}\Ebb_f\| \hat{\nu}(\mathbf{X}) - f \|_2^2 \geq \frac{\tau^2}{8c^2},$$
if $\tau$ satisfies $\log \Mcal_{\Fcal}^{\loc}(\tau, c) > 2N\tau^2/\alpha + 2\log 2$.
\end{lemma}
\begin{remark}
    In the proof of \cite[Lemma 6]{shrotriya2023revisiting}, there is a small typo. We will fix it in the appendix.
\end{remark}

Next, we establish a new minimax lower bound, drawing inspiration from the work of \cite{chen2017robustcovariancescattermatrix}. This lower bound is specifically designed to handle corruption schemes. Furthermore, in the following lemma, we demonstrate that this lower bound is tight up to constant factors.

\begin{lemma}[Lower bound in corruption schemes]\label{new lower bound}
Consider the corruption mechanism \(\mathcal{C}\) in the core problem setting: for a fixed \(\epsilon \in [0, \frac{1}{3}]\), a fraction \(\epsilon\) of the observations in \(\mathbf{\tilde{X}}\) are arbitrarily corrupted by some procedure \(\mathcal{C}\). Let \(\mathbf{X} = \mathcal{C}(\mathbf{\tilde{X}})\), where the \(i\)th coordinate is \(X_i = \mathcal{C}(\tilde{X}_i)\) if it is corrupted; otherwise, \(X_i = \tilde{X}_i\). Suppose $\epsilon \geq \frac{k}{N}$ with a constant $k > 7$, then the following lower bound holds:
$$
\inf_{\hat{f}} \sup_{f \in \Fcal} \sup_{\Ccal} \Ebb_f \left\|\hat{f}\left(\Ccal(\mathbf{\tilde{X}})\right) - f \right\|_2^2 \gtrsim \xi(\epsilon) \wedge d^2,
$$
where \(\xi(\epsilon) := \max_{f_1, f_2 \in \mathcal{F}, \, \TV(P_1, P_2) \leq \frac{\epsilon'}{1 - \epsilon'}} \{\|f_1 - f_2\|_2^2\}\), with $\epsilon' = \epsilon - \frac{1}{N}$, \(f_1\) and \(f_2\) denoting the densities of distributions \(P_1\) and \(P_2\), respectively. Furthermore, \(\xi(\epsilon) \lesssim \epsilon\). And $\xi(\epsilon) \gtrsim \epsilon^2$ when $\epsilon \lesssim \max\{d_\TV, d_{L_2}^2\}$, where \(d_{\TV}\) denotes the diameter of \(\mathcal{F}\) with respect to the total variation (TV) distance and \(d_{L_2}\) denotes the diameter of \(\mathcal{F}\) with respect to the \(L_2\)-norm.
\end{lemma}

%{\color{red} $(\max_{f, g \in \mathcal{F}, \operatorname{TV}(f,g) \leq \epsilon} \|f-g\|_2)^2$.}

%There is an inequality which says that 
%\begin{align*}
%    \|f-g\|_2^2 \asymp H^2(f,g) \lesssim \operatorname{TV}(f,g) \lesssim H(f,g) \asymp \|f-g\|_2
%\end{align*}

%Therefore the quantity $\epsilon \lesssim\max_{f, g \in \mathcal{F}, \operatorname{TV}(f,g) \leq \epsilon} \|f-g\|_2 \lesssim \sqrt{\epsilon}$
%{\color{red} the proof of the lemma above needs to be fixed. I think we need $\epsilon > 2/N$ or something like that?}
%{\color{red} If we are using $\mathcal{C}$ notation we should change the above lower bound to use the  $\mathcal{C}$ notation. The proof also needs to be changed.}

\begin{corollary}\label{corollary for new lower bound}
    Suppose \(\epsilon \geq \frac{k}{N}\) with a constant \(k > 7\). If \(\xi(\epsilon) \gtrsim \epsilon\), i.e., Condition \ref{Condition: L_2 and TV} holds, then the following lower bound holds:
$$
\inf_{\hat{f}} \sup_{f \in \Fcal} \sup_{\Ccal} \Ebb_f \left\|\hat{f}\left(\Ccal(\mathbf{\tilde{X}})\right) - f \right\|_2^2 \gtrsim \epsilon \wedge d^2.
$$
\end{corollary}

\subsection{Upper Bound}\label{Upper Bound Section}

%{\color{red} $G_j$ needs to be defined before you define $\psi(g,g', \mathbf{X})$. Also you should state that $k = 1/\epsilon$? We also need to figure out what is an upper bound on the $\epsilon$ beyond which our algo cannot work.}
Now, we turn to the upper bound. Let $X_{1}, \ldots, X_{N}\stackrel{\mbox{\scriptsize i.i.d.}}{\sim} f \in \mathcal{F}$ represent the $N$ observed (corrupted) samples and define $\mathbf{X} := (X_{1}, \ldots, X_{N})^{\top}$. Define $k := 1/(3\epsilon)$, where $\epsilon \leq \frac{1}{3}$ is the corruption rate.We divide the data into \(N/k = 3\epsilon N\) groups (assuming, for simplicity, that \(N/k\) is an integer) and denote the sets of indices for the data in each group by \(G_1, \dots, G_{N/k}\). We establish a criterion to compare two densities $g, g' \in \mathcal{F}$:
\begin{equation}\label{group log likelihood difference}
    \psi(g, g^{\prime}, \mathbf{X}) := \mathbbm{1}\left(\sum_{j =1}^{N/k} \mathbbm{1}\left(\sum_{i \in G_j} \log \frac{g(X_i)}{g^{\prime}(X_i)} > 0\right) \geq \frac{N}{2k}\right).
\end{equation}

If $\psi(g, g^{\prime}, \mathbf{X}) = 1$, then we say $g' \prec g$. At a high level, this means that we partition the data into $N/k$ groups. We consider $g$ to be ‘better’ than $g'$ if, in more than half of the groups, the difference in log-likelihoods between $g$ and $g'$ is greater than $0$, indicating that $g$ performs better than $g'$ in a majority of the groups. Note that there are at most $\epsilon N$ groups having corrupted data.

\begin{remark}
    If we let $k = \frac{1}{(2 + \gamma)\epsilon}$, where $\gamma > 0$ and $k$ is an integer, then it is possible to tolerate more than $1/3$ of the observations being corrupted.
\end{remark}

In fact, our universal estimator over \(\mathcal{F}\) is constructed as a likelihood-based estimator for \(f\). In particular:

\begin{remark}\label{nrmk:log-likelihood-well-defined}
    The log-likelihood difference \(\psi(g, g^{\prime}, \mathbf{X})\) in \eqref{group log likelihood difference} is well-defined. This is because, for each \(i \in [1, N]\), the individual random variables \(\log \frac{g(X_i)}{g^{\prime}(X_i)}\) are well-defined (as \(\alpha > 0\)) and bounded. Specifically, 
    \[
    -\infty < \log \frac{\alpha}{\beta} \leq \log \frac{g(X_i)}{g^{\prime}(X_i)} \leq \log \frac{\beta}{\alpha} < \infty,
    \]
    for each \(i \in [1, N]\).
\end{remark}

\begin{lemma}[Group log-likelihood-form difference concentration in $\Fcal$]\label{likelihood upper bound}
        Let $\delta > 0$ be arbitrary but assumed to be bigger than $\sqrt{\epsilon}$ up to absolute constant factors,  and let $X_{1}, \ldots, X_{N}\stackrel{\mbox{\scriptsize i.i.d.}}{\sim} f \in \mathcal{F}$ represent the $N$ observed (corrupted) samples and define $\mathbf{X} := (X_{1}, \ldots, X_{N})^{\top}$.
        For any two densities $g,
        g^{\prime} \in \Fcal$, let 
        $$\psi(g, g^{\prime}, \mathbf{X}) = \mathbbm{1}\left(\sum_{j =1}^{N/k} \mathbbm{1}\left(\sum_{i \in G_j} \log \frac{g(X_i)}{g^{\prime}(X_i)} > 0\right) \geq \frac{N}{2k}\right).$$ 

        We then have
    \begin{equation}
        \sup_{\substack{g, g^{\prime} \in \Fcal \colon \|g - g^{\prime}\|_{2} \geq C\delta, \\
                \|g^{\prime}-f\|_{2} \leq \delta}}\PP(\psi(g, g^{\prime}, \mathbf{X}) = 1)
        \leq
        \exp(-N C_{10}\delta^2),
        % \exp\parens{-2n \bigg(\sqrt{c(\alpha,\beta)}(C-1) - \sqrt{1 / \alpha}\bigg)^2\delta^2/\kappa},
    \end{equation}
    where
    \begin{align}
        C_{10} 
        & = \frac{L(\alpha, \beta, C)}{4} (\tfrac{1}{2}-\tfrac{\phi}{3}),\\
        C
         & >
        1 + \sqrt{1 / (\alpha c(\alpha,\beta))},
        % \text{ holds,}
         \\
        % \text{with }
        L(\alpha, \beta, C)
         & :=
        \frac{\left\{\sqrt{c(\alpha,\beta)} (C-1)  -  \sqrt{1 / \alpha}\right\}^{2}
        }{2( 2 K(\alpha, \beta) +\frac{2}{3} \log \beta / \alpha)},
    \end{align}
    with $K(\alpha, \beta) := \beta / (\alpha^{2} c(\alpha, \beta))$, and
    $c(\alpha, \beta)$ is as defined in Lemma \ref{lemma2 in shrotriya}. In the
    above $\PP$ is taken with respect to the true density function $f$, i.e.,
    $\PP = \PP_f$.
\end{lemma}

From Lemma \ref{likelihood upper bound}, we derive a key concentration result concerning a packing set in $\mathcal{F}$, as summarized in Lemma \ref{concentration likelihood}. Our estimator will be constructed using packing sets of $\mathcal{F}$, making Lemma \ref{concentration likelihood} an important tool for obtaining an upper bound for our minimax rate.

\begin{lemma}[Maximum likelihood concentration in $\Fcal$]\label{concentration likelihood}
Let $\nu_1,\dots,\nu_M$ be a maximal $\delta$-packing (covering) set of $\Fcal'\subseteq \Fcal$ with $f \in \Fcal'$ and $\delta \ge C_1\sqrt{\epsilon}$. Let $i^* \in \argmin_i T_i$, where 
\begin{align}
    T_i = \begin{cases}\max_{j \in E_i} \|\nu_i - \nu_j\|_2, & \text{ if \ } E_i := \left\{j\in [M]: \nu_i \prec \nu_j, \|\nu_i - \nu_j\|_2 \geq C \delta\right\} \text{ is not empty},\\
    0, & \mbox { otherwise}.
    \end{cases}
\end{align} 
Here, we denote $\nu_i \prec \nu_j$ as $\psi(\nu_j, \nu_i, \mathbf{X}) = 1$.

Then 
\begin{align*}
    \PP\left( \|\nu_{i^*} - f\|_2 \geq (C + 1) \delta\right) \leq M \exp\bigg( - C_{10} N \delta^2\bigg).
\end{align*}
We refer to \(\nu_{i^*}\) as the `best' density among \(\nu_1, \dots, \nu_M\).
\end{lemma}

Combining our criterion with the techniques from \cite{AkshayPrasadan2024} and \cite{shrotriya2023revisiting}, we now introduce the details of our algorithm to construct an estimator for $f \in \mathcal{F}$, as presented in Algorithm \ref{alg:sieve_estimator_tree}.

We will use the notation from \cite{AkshayPrasadan2024}. Given a node (i.e., element) \(u \in G\) for some set \(G\), we define the parent set $\mathcal{P}(u)$ as the set of nodes $u'$ with a directed edge from $u'$ to $u$ ($u' \rightarrow u$). We refer to a node $u$ as an offspring of a parent node $v$ if $v \in \mathcal{P}(u)$. For a node $v$, let $\mathcal{O}(v)$ denote the set of all offspring of $v$, i.e., $\mathcal{O}(v) = \{ u : v \in \mathcal{P}(u) \}$. We will call any $q \in \Fcal$ a point.

For our algorithm, we will first construct a pruned tree for the class $\mathcal{F}$ based on packing sets before observing any data. Then, we will traverse this tree from top to bottom.

Specifically, we first arbitrarily select a point $\nu_1 \in \mathcal{F}$ as our root node. We also have some sufficiently large $c > 0$, which is independent of the data, and define $C := \frac{c}{2} - 1$. The definition of $\Bar{J}$ is given in Theorem \ref{upper bound}. We fix some $\tilde{J} \geq \Bar{J}$. If $\tilde{J} = 1$, then this tree will consist of a single node. If $\tilde{J} = 2$, we treat $\nu_1$ as the first level of the tree, and then construct a maximal $\frac{d}{c}$-packing set of $\B_2(\nu_1, d) \cap \mathcal{F} = \mathcal{F}$. A directed edge is drawn from $\nu_1$ to each of these points to form level 2. Denote $\mathcal{L}(2)$ as the set of nodes forming the graph at level $2$. Otherwise, we continue the following steps until we have constructed the $\tilde{J}$th level of the tree:

For level $j \geq 3$, assuming we already have $\mathcal{L}(j-1)$, where $\mathcal{L}(j)$ denotes the set of nodes forming the pruned graph at level $j$ for $j > 2$, we iterate through each point $q$ in $\mathcal{L}(j-1)$. For each $q$, we construct a maximal $\frac{d}{2^{j-1}c}$-packing set of $\B_2(q, \frac{d}{2^{j-2}}) \cap \mathcal{F}$ and draw directed edges from $q$ to its associated packing set points. These newly generated points are denoted as candidate points in level $j$. Then, we apply a pruning step. Lexicographically order these candidate points at level $j$, i.e., $q^j_1, \dots, q^j_{M_j}$, where $M_j$ is the total number of these candidate points. Construct an ordered set $\Qcal_j = [q^j_1, \dots, q^j_{M_j}]$. While $\Qcal_j$ is not empty, pick the first element in $\Qcal_j$, say $q^j_l$. Construct a set $\Tcal_j(q^j_l) := \{q^j_k \in \Qcal_j : \|q^j_l - q^j_k\|_2 \leq \frac{d}{2^{j-1}c}, l \neq k\}$. For each $q^j_k \in \Tcal_j(q^j_l)$, remove the directed edge from $\Pcal(q^j_k)$ to $q^j_k$ and add a directed edge from $\Pcal(q^j_k)$ to $q^j_l$. Then, remove $q^j_k$ from the tree. Once finished iterating through the elements in $\Tcal_j(q^j_l)$, remove $\{q^j_l\} \cup \Tcal_j(q^j_l)$ from $\Qcal_j$. Once $\Qcal_j$ is empty, the resulting offspring nodes of points from level $j-1$ form level $j$, denoted as $\Lcal(j)$.

We continue this process until we have constructed the $\tilde{J}$th level of the tree.

Take the procedure of constructing $\Lcal(3)$ as an example: for each point $q$ in $\Lcal(2)$, we construct a maximal $\frac{d}{4c}$-packing set of $\B_2(q, \frac{d}{2}) \cap \Fcal$ and draw directed edges from $q$ to its corresponding maximal packing set points. After iterating through all the points in $\Lcal(2)$, lexicographically order these candidate points, i.e., $q^3_1, \dots, q^3_{M_3}$. Construct an ordered set $\Qcal_3 = [q^3_1, \dots, q^3_{M_3}]$. Assuming $\Qcal_3$ is not empty, consider the first element $q^3_1 \in \Qcal_3$. Construct a set $\Tcal_3(q^3_1) := \{q^3_k \in \Qcal_3 : \|q^3_1 - q^3_k\|_2 \leq \frac{d}{4c}, k \neq 1\}$, which may be empty. If $q^3_k \in \Tcal_3(q^3_1)$ for some $k$, then it cannot share the same parent as $q^3_1$ because the offspring of $\Pcal(q^3_1)$ form a $\frac{d}{4c}$-packing set. For each $q^3_k \in \Tcal_3(q^3_1)$, remove the directed edge from $\Pcal(q^3_k)$ to $q^3_k$ and add a directed edge from $\Pcal(q^3_k)$ to $q^3_1$. Now these $q^3_k$'s are points with no edges connected to them (meaning they are longer points in the tree graph). Remove $\{q^3_1\} \cup \Tcal_3(q^3_1)$ from $\Qcal_3$. Keep performing these pruning steps until $\Qcal_3$ is empty. Level $3$ nodes are then the set of offspring of nodes from level $2$, denoted as $\Lcal(3)$.

Next, we traverse this tree structure from the top layer to the bottom. We already have $\nu_1$. For the $j$th layer, we select the `best' density according to the criterion in Lemma \ref{concentration likelihood} from the offspring of $\nu_j$, based on the corrupted data $\mathbf{X}$ (note that this is the first time the algorithm encounters the data), and designate this density as $\nu_{j+1}$.

Finally, we set $\nu^* := \nu_{\tilde{J}}$ as the output of our algorithm.

\begin{algorithm}
\caption{Pruned tree structure construction of the multistage sieve estimator, $\nu^*(X)$, of $f \in \mathcal{F}$}
\label{alg:sieve_estimator_tree}
\textbf{Input:} A corrupted data $\mathbf{X} := (X_1, \dots, X_N)^T$ with corruption rate $\epsilon \in [0, \frac{1}{3}]$. Some sufficiently large $c > 0$ which is independent from the data and define $C := \frac{c}{2} - 1$. An integer number $\Bar{J}$ as defined in Theorem \ref{upper bound}. A root node $\nu_1$ randomly picked in $\mathcal{F}$. The diameter of $\mathcal{F}$, $d := \operatorname{diam}(\mathcal{F})$\;
Fixed some $\tilde{J} \geq \Bar{J}$\;
$j \gets 2$\;
$\Lcal(1) \gets \{\nu_1\}$\;

\If{$\tilde{J} = 1$}{
    We have a tree with only a single root node $\nu_1$\;
    \Return $\nu^*(\mathbf{X}) := \nu_1(\mathbf{X})$\;
}
\Else{
    \While{$j \leq \tilde{J}$}{
        Ordered set $\Qcal_j \gets \varnothing$\;
        \For{each point $q_i$ in $\Lcal(j - 1)$}{
            \If{j = 2}{
                Construct a maximal $\frac{d}{c}$-packing set $\Ucal_{q_i}$ of $\B_2(q_i, d) \cap \mathcal{F}$\;
            }
            \Else{Construct a maximal $\frac{d}{2^{j-1}c}$-packing set $\Ucal_{q_i}$ of $\B_2(q_i, \frac{d}{2^{j-2}}) \cap \mathcal{F}$\;}
            Add all the points in $\Ucal_{q_i}$ into the set $\Ocal(q_i)$ and $\Qcal_j$\;
            For each $p_k \in \Ucal_{q_i}$, set $\Pcal(p_k) = q_i$;
            %{\color{red} In Prasadan's paper, we do a maximal $d/c$ packing set on the second step. Then on the third step we do a $d/(4c)$ packing of the set $B(\cdot, d/2 \cap K)$ centered and any point $\cdot$ on the second level. So I think the above has to be changed for $j= 1$ case.}
        }
        \If{$j \geq 3$}{
            \While{$\Qcal_j$ not empty}{
                Pick first element $q^j_l \in \Qcal_j$ and construct a set $\Tcal_j(q^j_l) := \{q^j_k \in \Qcal_j : \|q^j_l - q^j_k\|_2 \leq \frac{d}{2^{j-1}c}, l \neq k\}$\;
                \For{$q^j_k \in \Tcal_j(q^j_l)$}{
                    Add $\{q^j_l\}$ to $\Ocal(\Pcal(q^j_k))$\;
                    %Add $\Pcal(q^j_k)$ to $\Pcal(q^j_l)$\;
                    %{\color{red} I don't think we need to add the parent of one node to the parent of the other node?}
                    
                    Remove $q^j_k$ from $\Ocal(\Pcal(q^j_k))$\;
                    %Set $\Pcal(q^j_k) = \varnothing$\;
                    %{\color{red} I don't think we need to set the parents of $q_k^j$ to nothing?}
                }
                Remove $\{q^j_l\} \cup \Tcal_j(q^j_l)$ from $\Qcal_j$\;
            }
        }
        \For{$q_i$ in $\Lcal(j - 1)$}{
            Add $\Ocal(q_i)$ into the set $\Lcal(j)$\;
        }
        $j \gets j+1$\;
    }
}

$j_0 \gets 2$\;
\While{$j_0 \leq \tilde{J}$}{
    Pick $\nu_{j_0}$ from the set $\Ocal(\nu_{j_0 - 1})$ such that $\nu_{j_0} := q_{i^*}$ with $i^*$ as defined in Lemma \ref{concentration likelihood}\;
    $j_0 \gets j_0 + 1$\;
}
\Return $\nu^*(\mathbf{X}) := \nu_{\tilde{J}}(\mathbf{X})$\;
\end{algorithm}

\cite{AkshayPrasadan2024} provides the proofs for the following two lemmas, which establish two important properties of the tree we constructed above.

\begin{lemma}[Upper bound for the cardinality of the offspring]\label{lemma:Upper bound for the cardinality of the offspring}
    Let $G$ be the pruned graph from above and assume $c > 2$. Then for any $j \geq 3$, $\mathcal{L}(J)$ forms a $d/(2^{j-2}c)$-covering of $\Fcal$ and a $d/(2^{j-1}c)$-packing of $\Fcal$. In addition, for each parent node $\Upsilon_{j-1}$ at level $j-1$, its offspring $\mathcal{O}(\Upsilon_{j-1})$ form a $d/(2^{j-2}c)$-covering of the set $\B_2(\Upsilon_{j-1}, d/2^{j-2}) \cap \Fcal$. Furthermore, the cardinality of $\mathcal{O}(\Upsilon_{j-1})$ is upper bounded by $\Mcal_\Fcal^{\loc}(d/2^{j-2}, 2c)$ for $j \geq 2$.
\end{lemma}

\begin{lemma}[Upper bound for the cardinality of $\mathcal{L}(j-1) \cap \B_2(f, d/2^{j-2})$]\label{lemma:Upper bound for the cardinality2}
    Pick any $f \in \Fcal$. Then for $J \geq 2$, $\mathcal{L}(j-1) \cap \B_2(f, d/2^{j-2})$ has cardinality upper bounded by $\Mcal_\Fcal^{\loc}(d/2^{j-2}, c) \leq \Mcal_\Fcal^{\loc}(d/2^{j-2}, 2c)$.
\end{lemma}

% By using Lemma \ref{lemma:Upper bound for the cardinality of the offspring}, Lemma \ref{lemma:Upper bound for the cardinality2} and Lemma $3.3$ in \cite{AkshayPrasadan2024}, we can derive the following two lemmas.

%\begin{lemma}[Concentration inequality]\label{lemma:concentration for large delta}
%    Let $\tau_J$ be defined as in Theorem \ref{upper bound}. Suppose $\tilde{J}$ satisfies \eqref{Equation:condition on J} and $\delta_{\tilde{J}}^2 := \frac{d^2}{2^{2(\tilde{J}-1)}(C+1)^2} \geq C_1^2 \epsilon$, where $C_1$ is a sufficiently large constant. Then for each such $1 \leq j \leq \tilde{J}$ we have
%    $$\PP(A_j) = \PP\left(\| f - \gamma_j \|_2 > \frac{d}{2^{j-1}}\right) \leq 2\exp(-N\tau_j^2/2)\mathbbm{1}(j > 1).$$
%\end{lemma}

%\begin{lemma}\label{lemma:lemma3.10 in prasadan}
%    Let $\tau_j$ be defined as in Theorem \ref{upper bound}. Suppose $\tilde{J}$ satisfies \eqref{Equation:condition on J} and also $\frac{d}{2^{\tilde{J}-1}(C+1)} \geq C_1 \sqrt{\epsilon}$ for some sufficient large constant $C_1$. Then if $\nu^*$ denotes the output after at least $\Bar{J}$ iterations, we have
%    $$\Ebb_f\|\nu^*-f\|_2^2 \leq \frac{(5c+2)^2}{C_{10}}{\tau_{\tilde{J}}}^2 + \frac{4(5c + 2)^2}{C_{10}} \mathbbm{1}(\Bar{J} > 1) \frac{1}{N}\exp(-N{\tau_{\tilde{J}}}^2/2).$$
%\end{lemma}

By Proposition 10 in \cite{shrotriya2023revisiting}, we know that $\nu^*(\mathbf{X})$ is a measurable function of the data with respect to the Borel $\sigma$-field on $\mathcal{F}$ in the $L_2$-metric topology. We can now derive a main theorem to establish the performance of $\nu^*(\mathbf{X})$, as shown in Theorem \ref{upper bound}.

\begin{theorem}[Upper bound rate for $\nu^*(X)$]\label{upper bound}
    Let $\nu^*$ be the output of the multistage sieve MLE which is run for $j$ steps where $j \geq \Bar{J}$. Here $\Bar{J}$ is defined as the maximal integer $J \in \mathbb{N}$, such that $\tau_J := \frac{\sqrt{C_{10}}d}{2^{(J-1)}(C+1)}$

satisfies
\begin{equation}\label{Equation:condition on J}
    N\tau_J^2 > 2\log \left[\Mcal_{\Fcal}^{\loc}\bigg( \tau_J\frac{c}{\sqrt{C_{10}}}, 2c\bigg)\right]^2 \vee \log 2,
\end{equation}
or $\Bar{J}=1$ if no such $J$ exists.

Then
\[
\mathbb{E}_f\|\nu^* - f\|_2^2 \lesssim \max\{ \tau_{\Bar{J}}^2, \epsilon \} \wedge d^2.
\]
We remind the reader that $c := 2(C+1)$ is the constant from the definition of local metric entropy, which is assumed to be sufficiently large. Here $C$ is assumed to satisfy (9), and $L(\alpha, \beta, C)$ is defined as per (10) in \cite{shrotriya2023revisiting}.
\end{theorem}
%{\color{red} as i mentioned in the proof, we need to put the square inside the log, not outside.}

%{\color{red} Note that our lower bound is using $c$ while the upper bound has $2c$ in the $M^{\operatorname{loc}}$. Hence we need to have a remark similar to the one in Prasadan and Neykov explaning why this difference in constants doesn't matter. I will copy the remark here so that it's easy for you to fix it:}

\begin{remark}  
In Definition \ref{def local metric}, which defines local metric entropy, we specify a constant, denoted by \(\tilde{c}\), so that \(\Mcal_\Fcal^{\loc}(\tau, \tilde{c})\) is computed using \(\tau/\tilde{c}\) packings of balls in \(K\). In the lower bound result of Lemma \ref{lowerbound in shrotriya}, we could have used \(\tilde{c} = 2c\) instead of \(c\), and the resulting bound would remain unchanged except for an absolute constant. Therefore, the local metric entropy parameter in the lower bound can be chosen to match the \(2c\) appearing in \eqref{Equation:condition on J} of Theorem \ref{upper bound}. Consequently, without loss of generality, we can assume that the same sufficiently large constant \(c\) appears in both the lower and upper bounds, and replace \eqref{Equation:condition on J} with the following: Let \(\Bar{J}\) be the maximal integer such that
\begin{equation} \label{New Equation:condition on J}
    \frac{N\tau_J^2}{\sigma^2} > 2\log \left[\Mcal_\Fcal^{\loc}\left(\tau_J\frac{c}{\sqrt{C_{10}}}, c\right)\right]^2 \vee \log 2,
\end{equation}
with \(\Bar{J} = 1\) if this condition never occurs.
\end{remark}

\subsection{Minimax Rate}\label{Minimax Rate Section}
Now we define 
\[
\tau^* := \sup\left\{\tau : N\tau^2 \leq \log \Mcal_{\Fcal}^{\loc}(\tau, c)\right\}.
\] 
We will use this and $\epsilon$ to achieve the optimal minimax rate. However, we first require the following lemma to handle the case when $\epsilon$ is small.

\begin{lemma}[Lower bound for optimal minimax rate]\label{lemma for minimax rate}
    Define $\tau^* := \sup\{\tau : N\tau^2 \leq \log \Mcal_{\Fcal}^{\loc}(\tau, c)\}$, where $c$ in the definition of local metric entropy is a sufficiently large absolute constant. When $\epsilon < \frac{k}{N}$ with a constant $k > 7$, we have 
    $$\max\{{\tau^*}^2 \wedge d^2, \epsilon \wedge d^2\} = {\tau^*}^2 \wedge d^2$$
    up to a constant.
\end{lemma}

Using the results from Lemma \ref{lowerbound in shrotriya}, Corollary \ref{corollary for new lower bound}, Theorem \ref{upper bound} and Lemma \ref{lemma for minimax rate}, we can formally demonstrate that the estimator obtained from our algorithm achieves the minimax rate if Condition \ref{Condition: L_2 and TV} holds, as shown in the following theorem.

\begin{theorem}[Minimax rate]\label{Minimax Rate}
    Define $\tau^* := \sup\{\tau : N\tau^2 \leq \log \Mcal_{\Fcal}^{\loc}(\tau, c)\}$, where $c$ in the definition of local metric entropy is a sufficiently large absolute constant. If Condition \ref{Condition: L_2 and TV} holds, the minimax rate is given by
\[
\max\{{\tau^*}^2 \wedge d^2, \epsilon \wedge d^2\} \text{ up to absolute constant factors.}
\]
\end{theorem}

\begin{remark}
    We can extend our results to loss functions in KL-divergence and the Hellinger metric, as stated in \cite{shrotriya2023revisiting}, because we have the `topological equivalence' of the KL-divergence and squared Hellinger metric with the squared $L_2$-metric on $\Fcal$ in Lemma \ref{lemma2 in shrotriya}.
\end{remark}

The following proposition relaxes the requirement of boundedness on $\mathcal{F}$, extending it from $\mathcal{F}_B^{[\alpha,\beta]}$ to $\mathcal{F}_B^{[0,\beta]}$, provided that there exists at least one $f_\alpha \in \mathcal{F}$ that is $\alpha$-lower bounded, where $\alpha > 0$.

%\textcolor{red}{It seems we do not need to change anything in the proof.}
\begin{proposition}[Extending results to $\mathcal{F}_B^{[0,\beta]}$]\label{Extending results to zero lower bound}
    Let $\mathcal{F} \subset \mathcal{F}_B^{[0,\beta]}$ be a star-shaped class of densities, with at least one $f_\alpha \in \mathcal{F}$ that is $\alpha$-lower bounded, with $\alpha > 0$. If Condition \ref{Condition: L_2 and TV} holds, then the minimax rate in the squared $L_2$-metric is $\max\{\tau^{*2} \wedge d^2, \epsilon \wedge d^2\}$, where $\tau^* := \sup \{\tau : N \tau^2 \leq \log M_\mathcal{F}^{\loc}(\tau, c)\}$. 
\end{proposition}

%{\color{red} include a comment somewhere that the above proof remains valid in our new setting.}

\begin{remark}
    The proof of Proposition \ref{Extending results to zero lower bound} in the case without corrupted data, as established in Proposition $13$ of \cite{shrotriya2023revisiting}, remains valid even in our case with corrupted data. We therefore omit the proof.
\end{remark}

%\section{Examples}\label{examples}

\section{Discussion}\label{discussion}
In this paper, we establish a new criterion for comparing the performance of two densities, \(g_1\) and \(g_2\), based on corrupted data. Using this criterion, we modify an algorithm to construct a density estimator within a star-shaped density class under corrupted data. We then leverage local metric entropy to derive the minimax lower and upper bounds for density estimation over a star-shaped density class consisting of densities that are uniformly bounded above and below (in the sup norm) in the presence of adversarially corrupted data. If Condition \ref{Condition: L_2 and TV} holds, these bounds are shown to match. %To the best of our knowledge, this is the first general result establishing the robust minimax rate for density estimation.

For future work, it is noteworthy that our estimator's performance relies on satisfying Condition \ref{Condition: L_2 and TV}, i.e. \ \(\xi(\epsilon) \gtrsim \epsilon\) when $\epsilon$ is `big', which is challenging to circumvent at present. We are keen to investigate whether it is possible to achieve the optimal minimax rate without this stringent condition. Moreover, this paper's analysis depends heavily on the equivalence between the KL divergence and the \(L_2\) loss to establish the lower bound. We are interested in examining whether similar results could be obtained using alternative loss functions. Another promising avenue for research is the development of an estimator for nonparametric regression within an adversarially corrupted data framework.

\bibliography{citation}
\bibliographystyle{plainnat}

\newpage
\appendix
\section{Minimax Lower Bounds}

\begin{lemma}[Lemma 6 in \cite{shrotriya2023revisiting}]
    Let $c > 0$ be fixed, and independent of the data samples $\mathbf{X}$. Then the minimax rate satisfies
$$\inf_{\hat{\nu}}\sup_{f\in\Fcal}\Ebb_f\| \hat{\nu}(\mathbf{X}) - f \|_2^2 \geq \frac{\tau^2}{8c^2},$$
if $\tau$ satisfies $\log \Mcal_{\Fcal}^{\loc}(\tau, c) > 2N\tau^2/\alpha + 2\log 2$.
\end{lemma}

\begin{proof}
    According to equation $(15.52)$ in \cite{wainwright2019high} and Lemma \ref{lemma2 in shrotriya}, for an arbitrary density \(\theta \in \Fcal\), consider the maximal packing set \(\{f^1, \dots, f^m\} \subset \Fcal \cap B(\theta, \tau)\) at a $L_2-$distance $\tau/c$. We have
    \[
    I(X_1; J) \leq \frac{1}{m}\sum_{j=1}^m d_{KL}(f^j \| \theta) \leq \max_{j \in [m]} d_{KL} (f^j \| \theta) \leq \max_{j \in [m]} \frac{1}{\alpha} \| f^j - \theta \|_2^2 \leq \frac{\tau^2}{\alpha}.
    \]
    Therefore, by Lemma \ref{Fano's inequality} (Fano's inequality), we have
    \[
    \inf_{\hat{\nu}} \sup_{f} \mathbb{E}_f \| \hat{\nu}(\mathbf{X}) - f \|_2^2 \geq \frac{\tau^2}{4c^2} \left( 1 - \frac{I(\mathbf{X}; J) + \log 2}{\log m} \right) \geq \frac{\tau^2}{4c^2} \left( 1 - \frac{N \tau^2 / \alpha + \log 2}{\log m} \right).
    \]
    Note that by the definition of $\Mcal_{\Fcal}^{\loc}(\tau, c)$, we have $\Mcal_{\Fcal}^{\loc}(\tau, c) := \sup_{f\in \Fcal}\Mcal(\tau/c, \Fcal \cap B(f, \tau))$. When \(\log m > \frac{2N\tau^2}{\alpha} + 2\log 2\), by taking the supremum over \(\theta\) for \(\log m\), we have \(\log \Mcal_{\Fcal}^{\loc}(\tau, c) > \frac{2N\tau^2}{\alpha} + 2\log 2\). Given that \(\theta \in \Fcal\) was chosen arbitrarily, this completes the proof.
\end{proof}

\begin{lemma}
Consider the corruption mechanism \(\mathcal{C}\) in the core problem setting: for a fixed \(\epsilon \in [0, \frac{1}{3}]\), a fraction \(\epsilon\) of the observations in \(\mathbf{\tilde{X}}\) are arbitrarily corrupted by some procedure \(\mathcal{C}\). Let \(\mathbf{X} = \mathcal{C}(\mathbf{\tilde{X}})\), where the \(i\)th coordinate is \(X_i = \mathcal{C}(\tilde{X}_i)\) if it is corrupted; otherwise, \(X_i = \tilde{X}_i\). Suppose $\epsilon \geq \frac{k}{N}$ with a constant $k > 7$, then the following lower bound holds:
$$
\inf_{\hat{f}} \sup_{f \in \Fcal} \sup_{\Ccal} \Ebb_f \left\|\hat{f}\left(\Ccal(\mathbf{\tilde{X}})\right) - f \right\|_2^2 \gtrsim \xi(\epsilon) \wedge d^2,
$$
where \(\xi(\epsilon) := \max_{f_1, f_2 \in \mathcal{F}, \, \TV(P_1, P_2) \leq \frac{\epsilon'}{1 - \epsilon'}} \{\|f_1 - f_2\|_2^2\}\), with $\epsilon' = \epsilon - \frac{1}{N}$, \(f_1\) and \(f_2\) denoting the densities of distributions \(P_1\) and \(P_2\), respectively. Furthermore, \(\xi(\epsilon) \lesssim \epsilon\). And $\xi(\epsilon) \gtrsim \epsilon^2$ when $\epsilon \lesssim \max\{d_\TV, d_{L_2}^2\}$, where \(d_{\TV}\) denotes the diameter of \(\mathcal{F}\) with respect to the total variation (TV) distance and \(d_{L_2}\) denotes the diameter of \(\mathcal{F}\) with respect to the \(L_2\)-norm.
\end{lemma}

\begin{proof}
Given that a portion of the data is corrupted, these data points originate from some arbitrary corrupting process \(q\). We define \(\mathcal{P}\) as the set of all possible joint distributions \(\mathbb{P}_{f,q,\epsilon}\) based on data \(\mathbf{\tilde{X}}\) with a corruption rate \(\epsilon\), where the uncorrupted data follows some density \(f \in \mathcal{F}\) and the corrupted data follows \(q\). Specifically, \(\mathcal{P} := \{\mathbb{P}_{f,q,\epsilon}: f \in \mathcal{F}, q \text{ is any arbitrary corrupting process}, \epsilon \in [0, \frac{1}{3}]\}\). We will find two distributions \(\mathbb{P}_{f_0, q_0, \Bar{\epsilon}}, \mathbb{P}_{f_1, q_1, \Bar{\epsilon}} \in \mathcal{P}\) and use Le Cam's two-point lemma to establish that, for \(\|f_0 - f_1\|_2^2 = \omega(\epsilon)\), \(\inf_{\hat{f}} \sup_{f \in \mathcal{F}} \mathbb{E}_{f} \|\hat{f} - f \|_2^2 \geq \frac{\omega(\epsilon)}{2}\left(1-\|\mathbb{P}_{f_0, q_0, \Bar{\epsilon}} - \mathbb{P}_{f_1, q_1, \Bar{\epsilon}}\|_{\mathrm{TV}}\right)\).

Inspired by the proof in \cite{AkshayPrasadan2024} and similar to the approach in the proof of Theorem 5.1 in \cite{chen2017robustcovariancescattermatrix}, consider two probability measures \(P_1\) and \(P_2\) such that \(\TV(P_1, P_2) \leq \frac{\epsilon'}{1-\epsilon'}\). where $\epsilon' = \epsilon - \frac{1}{N} \leq \frac{1}{3} - \frac{1}{N}$. Recall that by assumption, \(\epsilon \leq \frac{1}{3}\). Note that $\epsilon' < \frac{\epsilon'}{1-\epsilon'} < \frac{3}{2}\epsilon'$. And since $\epsilon \geq \frac{k}{N}$, where $k > 7$, we have $\frac{1}{\epsilon N} \leq \frac{1}{k}$ and $\epsilon > \epsilon' = \epsilon(1-\frac{1}{\epsilon N}) \geq \frac{(k-1)\epsilon}{k}$. Thus, $\frac{\epsilon'}{1-\epsilon'} \asymp \epsilon$.

%{\color{red} Here's an easier way to prove this inequality: Since $f_1, f_2 \in \Fcal \subset \Fcal^{[\alpha, \beta]}$, this implies that $|f_1 - f_2| \leq M$ for some constant $M$. Then we have $|\int (f_1 - f_2)^2| \leq M\int|f_1 - f_2|$. Therefore, we have $\|f_1 - f_2\|_2^2 \lesssim \TV(f_1, f_2)$. And by using Holder's inequality, we have $\int |f_1 - f_2| \leq (\int (f_1 - f_2)^2)^{1/2}$. Therefore, we have $\TV(f_1, f_2) \lesssim \|f_1 - f_2\|_2$.}

Since $f_1, f_2 \in \Fcal \subset \Fcal^{[\alpha, \beta]}$, this implies that $|f_1 - f_2| \leq M$ for some constant $M$. Then we have $|\int (f_1 - f_2)^2d\mu| \leq M\int|f_1 - f_2|d\mu$. Therefore, we have $\|f_1 - f_2\|_2^2 \lesssim \TV(f_1, f_2)$,
where \(f_1\) and \(f_2\) are the densities of \(P_1\) and \(P_2\), respectively. 
%We also have 
%\begin{align*}
%    \|f_1-f_2\|_2^2 \asymp H^2(f_1,f_2) \lesssim \operatorname{TV}(f_1,f_2) \lesssim H(f_1,f_2) \asymp \|f_1-f_2\|_2,
%\end{align*}
%where $H(f_1, f_2)$ is the Hellinger distance between $f_1$ and $f_2$. 

Denote \[\xi(\epsilon) := \max_{f_1, f_2 \in \mathcal{F}, \, \TV(P_1, P_2) \leq \frac{\epsilon'}{1 - \epsilon'}} \{\|f_1 - f_2\|_2^2\}.\] Thus we have $\xi(\epsilon) \lesssim \epsilon$. 

Another way to show that is by Pinsker's inequality and Lemma \ref{lemma2 in shrotriya}, we have
\[
[\TV(P_1, P_2)]^2 \leq \frac{1}{2} d_{KL}(f_1 \| f_2) \leq \frac{1}{2\alpha} \| f_1 - f_2 \|_2^2,
\]

We also have 
\begin{equation}\label{inequality of TV, Hellinger and L_2}
    \|f_1-f_2\|_2^2 \asymp H^2(f_1,f_2) \lesssim \operatorname{TV}(f_1,f_2) \lesssim H(f_1,f_2) \asymp \|f_1-f_2\|_2,
\end{equation}
where $H(f_1, f_2)$ is the Hellinger distance between $f_1$ and $f_2$. 

Note that if \(\epsilon \lesssim d_{\TV}\), where \(d_{\TV}\) denotes the diameter of \(\mathcal{F}\) with respect to the total variation (TV) distance, we can construct two densities \(f_1\) and \(f_2\) such that \(\TV(f_1, f_2) = \epsilon\). Consequently, \(\xi(\epsilon) \gtrsim \epsilon^2\) since \(\|f_1 - f_2\|_2 \gtrsim \TV(f_1, f_2)\) by~\eqref{inequality of TV, Hellinger and L_2}.

On the other hand, if \(\epsilon \lesssim d_{L_2}^2\), where \(d_{L_2}\) denotes the diameter of \(\mathcal{F}\) with respect to the \(L_2\)-norm, we have \(d_{L_2}^2 \lesssim d_{\TV}\). Therefore, \(\xi(\epsilon) \gtrsim \epsilon^2\) by~\eqref{inequality of TV, Hellinger and L_2}.

Suppose for some \(\epsilon'' \leq \epsilon'\), \(\TV(P_1, P_2) = \frac{\epsilon''}{1-\epsilon''}\). We define the densities \(q_1\) and \(q_2\) by
\[
q_1 = \frac{(f_2 - f_1)\mathbbm{1}(f_2 \geq f_1)}{\TV(P_1, P_2)}, \quad q_2 = \frac{(f_1 - f_2)\mathbbm{1}(f_1 \geq f_2)}{\TV(P_1, P_2)}.
\]
By the proof of Theorem $5.1$ in \cite{chen2017robustcovariancescattermatrix}, we know that \(q_1\) and \(q_2\) are probability densities.

Now, define the measures \(Q_1\) and \(Q_2\) with densities \(q_1\) and \(q_2\), respectively, i.e.,
\[
\frac{\d Q_1}{\d (P_1 + P_2)} = q_1, \quad \frac{\d Q_2}{\d(P_1 + P_2)} = q_2.
\]

Consider two mixture measures \((1-\epsilon'')P_1 + \epsilon''Q_1\) and \((1-\epsilon'')P_2 + \epsilon''Q_2\). We have
\begin{align*}
    \d \left((1-\epsilon'')P_1 + \epsilon''Q_1\right) &= (1-\epsilon'')f_1 + (1-\epsilon'')(f_2 - f_1)\mathbbm{1}(f_2 \geq f_1) \\
    &= (1-\epsilon'')f_1\mathbbm{1}(f_1 \geq f_2) + (1-\epsilon'')f_2\mathbbm{1}(f_2 \geq f_1) \\
    &= (1-\epsilon'')(f_1 - f_2)\mathbbm{1}(f_1 \geq f_2) + (1-\epsilon'')f_2 \\
    &= \d\left((1-\epsilon'')P_2 + \epsilon'' Q_2\right).
\end{align*}

Thus, we have
\[
\TV\left(((1-\epsilon'')P_1 + \epsilon''Q_1)^{\otimes N}, ((1-\epsilon'')P_2 + \epsilon''Q_2)^{\otimes N}\right) = 0.
\]

Now, let's expand the total variation (TV) distance. Define the set family \[\binom{[N]}{s} = \{ \{i_1, \dots, i_s\}: i_1, \dots, i_s \in [N], i_1 < i_2 < \dots < i_s \}.\] Introduce the shorthand notation:
\[
u(\mathbf{X}) := \prod_{i \in [N] \setminus I} f_1(x_i) \prod_{i \in I} q_1(x_i) - \prod_{i \in [N] \setminus I} f_2(x_i) \prod_{i \in I} q_2(x_i).
\]

We have:
\begin{align*}
0 &= \TV\left(((1-\epsilon'')P_1 + \epsilon''Q_1)^{\otimes N}, ((1-\epsilon'')P_2 + \epsilon''Q_2)^{\otimes N}\right) \\
&= \frac{1}{2} \int \left| \prod_{i \in [N]} \left[(1-\epsilon'')f_1(x_i) + \epsilon'' q_1(x_i)\right] - \prod_{i \in [N]} \left[(1-\epsilon'')f_2(x_i) + \epsilon'' q_2(x_i)\right] \right| \prod \d x_i \\
&= \frac{1}{2} \int \left| \sum_{s=0}^N \sum_{I \in \binom{[N]}{s}} (1-\epsilon'')^{N-s} (\epsilon'')^{s} u(\mathbf{X}) \right| \prod \d x_i \\
&\geq \frac{1}{2} \int \left| \sum_{s=0}^{2\epsilon''N} \sum_{I \in \binom{[N]}{s}} (1-\epsilon'')^{N-s} (\epsilon'')^{s} u(\mathbf{X}) \right| \prod \d x_i \\
&\quad - \frac{1}{2} \int \left| \sum_{s=2\epsilon''N+1}^N \sum_{I \in \binom{[N]}{s}} (1-\epsilon'')^{N-s} (\epsilon'')^{s} u(\mathbf{X}) \right| \prod \d x_i \\
&\geq \frac{1}{2} \int \left| \sum_{s=0}^{2\epsilon''N} \sum_{I \in \binom{[N]}{s}} (1-\epsilon'')^{N-s} (\epsilon'')^{s} u(\mathbf{X}) \right| \prod \d x_i - \PP\left(\Bin(N, \epsilon'') > 2\epsilon''N\right) \\
&\geq \frac{1}{2} \int \left| \sum_{s=0}^{2\epsilon''N} \sum_{I \in \binom{[N]}{s}} (1-\epsilon'')^{N-s} (\epsilon'')^{s} u(\mathbf{X}) \right| \prod \d x_i - \exp(-ND(2\epsilon''\|\epsilon'')).
\end{align*}

The last inequality holds by a binomial Chernoff bound from Section 1.3 in \cite{DubhashiPanconesi2009}. Recalling that \(D(q \| p) = q \log \frac{q}{p} + (1-q) \log \frac{1-q}{1-p}\), we have $D(2\epsilon'' \| \epsilon'') = 2\epsilon''\log2 + (1-2\epsilon'')\log(1 - \frac{\epsilon''}{1-\epsilon''})$.

By Taylor's expansion, we have 
\[
\log(1 - x) = - x - \frac{x^2}{2} - \frac{x^3}{3(1 - u)^3},
\] 
where \( u \in (0, x) \) and \( |x| < 1 \). When \( 0 \leq x \leq \frac{1}{2} \), we obtain 
\[
1 + \frac{2x}{3(1 - u)^3} < 1 + \frac{2x}{3(1 - x)^3} \leq \frac{11}{3} =: c_0.
\] 
This implies that 
\[
\frac{x^3}{3(1 - u)^3} \leq \frac{x^2}{2}(c_0 - 1),
\]
i.e., 
\[
- \frac{x^2}{2} - \frac{x^3}{3(1 - u)^3} \geq - \frac{x^2}{2}c_0.
\]
Note that since \( \epsilon'' < \frac{1}{3} \), we have \( 0 \leq \frac{\epsilon''}{1 - \epsilon''} < \frac{1}{2} \). Therefore, 
\[
\log\left(1 - \frac{\epsilon''}{1 - \epsilon''}\right) \geq - \frac{\epsilon''}{1 - \epsilon''} - \frac{1}{2}\left(\frac{\epsilon''}{1 - \epsilon''}\right)^2 c_0.
\]

Denote 
\[
\kappa(\epsilon'') = 2\epsilon'' \log 2 - (1 - 2\epsilon'') \frac{\epsilon''}{1 - \epsilon''} - (1 - 2\epsilon'')/2\left(\frac{\epsilon''}{1 - \epsilon''}\right)^2 \cdot c_0.
\] 
Now we minimize the function \( \frac{\kappa(x)}{x} \) given $x \in [0, 1/3]$. The derivative of this function is 
\[
-\frac{27x-5}{6(x - 1)^3} > 0.
\]
Thus, the minimizer of \( \frac{\kappa(x)}{x} \) occurs at \( x^* = \frac{1}{3} \), and we have 
\[
\frac{\kappa(x^*)}{x^*} = 2 \log 2 - \frac{23}{24} > 0.3.
\] 
This implies that 
\[
D(2\epsilon'' \| \epsilon'') > 0.3 \epsilon''
\]
and 
\[
\PP\left(\Bin(N, \epsilon'') > 2\epsilon'' N\right) \leq \exp(-0.3 N \epsilon'').
\] 
Thus, if we choose \( \epsilon'' > \frac{3}{N} \), we will have 
\[
\PP\left(\Bin(N, \epsilon'') > 2\epsilon'' N\right) \leq \exp(-0.9).
\] 
Furthermore, we require \( 2\epsilon'' < \epsilon \) so that our construction below is within \( \Pcal \). Note that we can always find some $\epsilon''$ such that \( \epsilon / 2 > \epsilon'' > \frac{3}{N} \) since the only constraint is \( \epsilon'' \leq \epsilon' = \epsilon - \frac{1}{N} \) and \( \epsilon \geq \frac{k}{N} \) with \( k > 7 \).

Therefore, we have
\begin{align*}
\mu_{\TV}
&:= \frac{\frac{1}{2} \int \left| \sum_{s=0}^{2\epsilon''N} \sum_{I \in \binom{[N]}{s}} (1-\epsilon'')^s (\epsilon'')^{N-s} u(\mathbf{X})\right| \prod \d x_i}{\PP\left(\Bin(N, \epsilon'') \leq 2\epsilon''N\right)} \\
& \leq \frac{\exp(-0.9)}{1 - \exp(-0.9)} < 1.
\end{align*}

Here, \(\mu_{\TV}\) is the total variation between the mixture measures corrupting less than or equal to \( 2\epsilon'' < \epsilon\) of the observations at random, i.e. \(\PP_{f_1, q_1, 2\epsilon''} \in \Pcal\) and \(\PP_{f_2, q_2, 2\epsilon''} \in \Pcal\)  (with densities \(h_0\) and \(h_1\), respectively). The following two densities are valid probability densities:

\[
h_0(x_1, \ldots, x_N) = \frac{\sum_{s=0}^{2\epsilon''N} \sum_{I \in \binom{[N]}{s}} (1-\epsilon'')^s (\epsilon'')^{N-s} \prod_{i \in I} f_1(x_i) \prod_{i \in [N] \setminus I} q_1(x_i)}{\PP\left(\Bin\left(N, \epsilon''\right) \leq 2\epsilon''N\right)}
\]

\[
h_1(x_1, \ldots, x_N) = \frac{\sum_{s=0}^{2\epsilon''N} \sum_{I \in \binom{[N]}{s}} (1-\epsilon'')^s (\epsilon'')^{N-s} \prod_{i \in I} f_2(x_i) \prod_{i \in [N] \setminus I} q_2(x_i)}{\PP\left(\Bin\left(N, \epsilon''\right) \leq 2\epsilon''N\right)}
\]

Therefore, according to Le Cam's two-point lemma in Chapter 15 of \cite{wainwright2019high}, and since \(\mu_{\TV} \leq \frac{\exp(-0.9)}{1 - \exp(-0.9)} < 1\), we have
\[
\inf_{\hat{f}} \sup_{f \in \Fcal} \Ebb_{f} \|\hat{f} - f \|_2^2 \geq \left(\xi(\epsilon) \wedge d^2\right) \frac{1}{2}(1-\mu_{\TV}) \gtrsim \xi(\epsilon) \wedge d^2.
\]

Thus, we have
$$
\inf_{\hat{f}} \sup_{f \in \Fcal} \sup_{\Ccal} \Ebb_f \left\|\hat{f}\left(\Ccal(\mathbf{\tilde{X}})\right) - f \right\|_2^2 \gtrsim \xi(\epsilon) \wedge d^2.
$$
\end{proof}

\section{Minimax Upper Bound}
\begin{lemma}
        Let $\delta > 0$ be arbitrary but assumed to be bigger than $\sqrt{\epsilon}$ up to absolute constant factors,  and let $X_{1}, \ldots, X_{N}\stackrel{\mbox{\scriptsize i.i.d.}}{\sim} f \in \mathcal{F}$ represent the $N$ observed (corrupted) samples and define $\mathbf{X} := (X_{1}, \ldots, X_{N})^{\top}$.
        For any two densities $g,
        g^{\prime} \in \Fcal$, let 
        $$\psi(g, g^{\prime}, \mathbf{X}) = \mathbbm{1}\left(\sum_{j =1}^{N/k} \mathbbm{1}\left(\sum_{i \in G_j} \log \frac{g(X_i)}{g^{\prime}(X_i)} > 0\right) \geq \frac{N}{2k}\right).$$ 

        We then have
    \begin{equation}
        \sup_{\substack{g, g^{\prime} \in \Fcal \colon \|g - g^{\prime}\|_{2} \geq C\delta, \\
                \|g^{\prime}-f\|_{2} \leq \delta}}\PP(\psi(g, g^{\prime}, \mathbf{X}) = 1)
        \leq
        \exp(-N C_{10}\delta^2),
        % \exp\parens{-2n \bigg(\sqrt{c(\alpha,\beta)}(C-1) - \sqrt{1 / \alpha}\bigg)^2\delta^2/\kappa},
    \end{equation}
    where
    \begin{align}
        C_{10} 
        & = \frac{L(\alpha, \beta, C)}{4} (\tfrac{1}{2}-\tfrac{\phi}{3}),\\
        C
         & >
        1 + \sqrt{1 / (\alpha c(\alpha,\beta))},
        % \text{ holds,}
         \\
        % \text{with }
        L(\alpha, \beta, C)
         & :=
        \frac{\left\{\sqrt{c(\alpha,\beta)} (C-1)  -  \sqrt{1 / \alpha}\right\}^{2}
        }{2( 2 K(\alpha, \beta) +\frac{2}{3} \log \beta / \alpha)},
    \end{align}
    with $K(\alpha, \beta) := \beta / (\alpha^{2} c(\alpha, \beta))$, and
    $c(\alpha, \beta)$ is as defined in Lemma \ref{lemma2 in shrotriya}. In the
    above $\PP$ is taken with respect to the true density function $f$, i.e.,
    $\PP = \PP_f$.
\end{lemma}

\begin{proof}
By Lemma 7 of \cite{shrotriya2023revisiting}, we know that for all \(j \in [N/k]\):
\begin{align*}
    \PP\left(\sum_{i \in G_j} \log \frac{g(\tilde{X}_i)}{g'(\tilde{X}_i)} > 0\right) \leq \exp\left(-k L(\alpha, \beta, C) \delta^2\right).
\end{align*}

Given that \(\delta \gtrsim \sqrt{\epsilon}\) for a sufficiently large constant, and recalling that \(k = \frac{1}{3\epsilon}\), the above probability can be made smaller than \(\varrho := \frac{1}{4} \exp\left(-k \frac{L(\alpha, \beta, C)}{2} \delta^2\right)\). Let \(\phi \in \left(0, \frac{3}{2}\right)\) be a constant.

Define \(A_j\) as the event \(\sum_{i \in G_j} \log \frac{g(\tilde{X}_i)}{g'(\tilde{X}_i)} > 0\) and \(B_j\) as the event \(\sum_{i \in G_j} \log \frac{g(X_i)}{g'(X_i)} > 0\). Let \(\tilde{\psi} = 1\) denote the event that at least \(\frac{3}{2}\epsilon N - \epsilon N\phi(1-\varrho)\) of the \(A_j\) events occur, and let \(\psi = 1\) denote the event that at least \(\frac{3}{2} \epsilon N\) of the \(B_j\) events occur. We assume that \(\phi(1-\varrho) > 1\). 

We aim to prove \(\PP(\psi = 1) \leq \PP(\tilde{\psi} = 1)\). If \(\tilde{\psi} = 0\), then no more than \(\frac{3}{2}\epsilon N - \epsilon N\phi(1-\varrho)\) of the \(A_j\) events occur. Then, no more than 
\[
\left(\frac{3}{2}\epsilon N - \epsilon N\phi(1-\varrho) + \epsilon N\right) = \frac{3}{2}\epsilon N + \epsilon N(1 - \phi(1-\varrho)) \leq \frac{3}{2}\epsilon N
\]
of the \(B_j\) events occur, i.e., $\tilde{\psi} = 0$ implies $\psi = 0$. Therefore, \(\PP(\psi = 0) \geq \PP(\tilde{\psi} = 0)\), proving our claim that \(\PP(\psi = 1) \leq \PP(\tilde{\psi} = 1)\). 

Observe that \(\tilde{\psi} = 1\) implies that no more than \(\frac{3}{2}\epsilon N + \epsilon N\phi(1-\varrho)\) of the \(A_j^c\) events occur. Recall that \(\PP(A_j^c) \geq 1 - \varrho\). Set \(p = 1 - \varrho\) and \(\chi = \frac{1}{2} - \varrho - \frac{\phi(1-\varrho)}{3}\), so that \(p - \chi = \frac{1}{2} + \frac{\phi(1-\varrho)}{3}\) and \(p \leq \PP(A_j^c)\).

Following earlier arguments, we have:
\begin{align*}
    \PP(\tilde{\psi} = 1) &\leq \PP\left(\mathrm{Bin}(3\epsilon N, \PP(A_j^c)) \leq \frac{3}{2}\epsilon N + \epsilon N\phi(1-\varrho)\right) \\
    &=  \PP\left(\mathrm{Bin}(3\epsilon N, \PP(A_j^c)) \leq 3\epsilon N(p - \chi)\right) \\
    &\leq  \PP\left(\mathrm{Bin}(3\epsilon N, p) \leq 3\epsilon N(p - \chi)\right) \\
    &\leq \exp\left(-3\epsilon N D(p - \chi \| p)\right),
\end{align*}
where we used Theorem 1(a) in \cite{KlenkeMattner2010} for the stochastic dominance property of binomials and then applied a binomial Chernoff bound from Section 1.3 in \cite{DubhashiPanconesi2009}. 

Recalling that \(D(q \| p) = q \log \frac{q}{p} + (1-q) \log \frac{1-q}{1-p}\), we compute:
\begin{align}
    D(p - \chi \| p) &= D\left(\frac{1}{2} + \frac{\phi(1-\varrho)}{3} \middle\| 1 - \varrho\right) \notag \\
    &= \left(\frac{1}{2} + \frac{\phi(1-\varrho)}{3}\right) \log \frac{\frac{1}{2} + \frac{\phi(1-\varrho)}{3}}{1 - \varrho} + \left(\frac{1}{2} - \frac{\phi(1-\varrho)}{3}\right) \log \frac{\frac{1}{2} - \frac{\phi(1-\varrho)}{3}}{\varrho} \notag \\
    &\geq \left(\frac{1}{2} + \frac{\phi}{3}\right) \log \frac{\frac{1}{2} + \frac{\phi(1-\varrho)}{3}}{1 - \varrho} + \left(\frac{1}{2} - \frac{\phi}{3}\right) \log \frac{\frac{1}{2} - \frac{\phi(1-\varrho)}{3}}{\varrho}  \label{eq:SG:setting_rho_0} \\
    &\geq \left(\frac{1}{2} + \frac{\phi}{3}\right) \log \left(\frac{1}{2} + \frac{\phi}{3}\right) + \left(\frac{1}{2} - \frac{\phi}{3}\right) \log \frac{\frac{1}{2} - \frac{\phi}{3}}{\varrho} \label{eq:SG:deriv_positive} \\
    &= g(\phi) + \left(\frac{1}{2} - \frac{\phi}{3}\right) \log \frac{1}{\varrho}, \notag
\end{align}
where \(g(\phi) = \left(\frac{1}{2} + \frac{\phi}{3}\right) \log \left(\frac{1}{2} + \frac{\phi}{3}\right) + \left(\frac{1}{2} - \frac{\phi}{3}\right) \log \left(\frac{1}{2} - \frac{\phi}{3}\right)\) and \(g(\phi) < 0\). 

To justify inequality \eqref{eq:SG:setting_rho_0}, note that \(6\varrho^2 - 7\varrho + 3 > 0\) always holds, which is equivalent to \(\frac{3 - 3\varrho}{4 - 6\varrho} > \varrho\). Recall that \(\phi \in \left(0, \frac{3}{2}\right)\), \(\varrho \in \left[0, \frac{1}{4}\right)\), and \(\phi(1-\varrho) > 1\). Thus, we have \(\frac{3 - 2\phi}{6 - 2\phi} \in \left(\frac{3 - 3\varrho}{4 - 6\varrho}, 1\right)\). Then \(\varrho < \frac{3 - 2\phi}{6 - 2\phi}\), which implies \(\frac{\frac{1}{2} + \frac{\phi(1-\varrho)}{3}}{1 - \varrho} < 1 < \frac{\frac{1}{2} - \frac{\phi(1-\varrho)}{3}}{\varrho}\). Hence, the first logarithm term is negative while the second logarithm term is positive, leading to our inequality when setting \(\varrho = 0\) in the constants in front of each respective logarithm term. For \eqref{eq:SG:deriv_positive}, we use the fact that \(\frac{d}{d\varrho} \frac{\frac{1}{2} + \frac{\phi(1-\varrho)}{3}}{1 - \varrho} = \frac{1}{2(1 - \varrho)^2} > 0\), so the first logarithm term increases with \(\varrho\). Our bound follows from setting \(\varrho = 0\) inside the first logarithm term and also setting \(\varrho\) in the numerator of the second logarithm term to 0.  

Since we have \(\phi(1-\varrho) > 1\), this is equivalent to \(\log(\varrho) < \log\left(1 - \frac{1}{\phi}\right)\). Recall that \(\delta \gtrsim \sqrt{\epsilon}\) for a sufficiently large constant, so \(\varrho\) can be made smaller than any fixed positive constant. Therefore, we have
\[
\log(\varrho) \leq \min\left\{\log\left(1 - \frac{1}{\phi}\right), \frac{2g(\phi)}{\frac{1}{2} - \frac{\phi}{3}}\right\}.
\]
Thus, the inequality
\[
\left(\frac{1}{2} - \frac{\phi}{3}\right) \log\left(\frac{1}{\varrho}\right) \geq -2g(\phi) > 0
\]
holds, and we have
\[
\begin{aligned}
    D(p - \chi \| p) &\geq \frac{1}{2}\left(\frac{1}{2} - \frac{\phi}{3}\right) \log\left(\frac{1}{\varrho}\right) \\
    &= \frac{L(\alpha, \beta, C)}{4} k \delta^2 \left(\frac{1}{2} - \frac{\phi}{3}\right) + \frac{\log 4}{2}\left(\frac{1}{2} - \frac{\phi}{3}\right) \\
    &\geq \frac{L(\alpha, \beta, C)}{4} k \delta^2 \left(\frac{1}{2} - \frac{\phi}{3}\right).
\end{aligned}
\]

Therefore, noting that \(3\epsilon \cdot k = 1\),
\begin{align*}
     \PP(\tilde{\psi} = 1) &\leq \exp\left(-3\epsilon N \cdot D(p - \chi \| p)\right) \\
     &\leq \exp\left(-3\epsilon N \cdot \frac{L(\alpha, \beta, C)}{4} k \delta^2 \left(\frac{1}{2} - \frac{\phi}{3}\right)\right) \\
     &= \exp(-N C_{10} \delta^2),
\end{align*}
where \(C_{10} = \frac{L(\alpha, \beta, C)}{4} \left(\frac{1}{2} - \frac{\phi}{3}\right)\).
\end{proof}

\begin{lemma} 
Let $\nu_1,\dots,\nu_M$ be a maximal $\delta$-packing (covering) set of $\Fcal'\subseteq \Fcal$ with $f \in \Fcal'$ and $\delta \ge C_1\sqrt{\epsilon}$. Let $i^* \in \argmin_i T_i$, where 
\begin{align}
    T_i = \begin{cases}\max_{j \in E_i} \|\nu_i - \nu_j\|_2, & \text{ if \ } E_i := \left\{j\in [M]: \nu_i \prec \nu_j, \|\nu_i - \nu_j\|_2 \geq C \delta\right\} \text{ is not empty},\\
    0, & \mbox { otherwise}.
    \end{cases}
\end{align} 
Here, we denote $\nu_i \prec \nu_j$ as $\psi(\nu_j, \nu_i, \mathbf{X}) = 1$.

Then 
\begin{align*}
    \PP\left( \|\nu_{i^*} - f\|_2 \geq (C + 1) \delta\right) \leq M \exp\bigg( - C_{10} N \delta^2\bigg).
\end{align*}
We refer to \(\nu_{i^*}\) as the `best' density among \(\nu_1, \dots, \nu_M\).
\end{lemma}

\begin{proof}
    Without loss of generality, let \(\nu_1\) be the closest point to \(f\). Since \(\Fcal'\) is \(\delta\)-packed, we have \(\|\nu_1 - f\|_2 \leq \delta\) (a maximal \(\delta\)-packing set of \(\Fcal'\) is also a \(\delta\)-covering set of \(\Fcal'\)). By definition, \(\max(T_i, T_j) \geq \|\nu_i - \nu_j\|_2\) whenever \(\|\nu_i - \nu_j\|_2 \geq C \delta\). Consequently,
    \begin{align*}
        \mathbbm{1}_{\|\nu_{i^*} - \nu_1\|_2 \geq C \delta} \leq \mathbbm{1}_{\max(T_{i^*}, T_1) \geq C \delta} = \mathbbm{1}_{T_1 \geq C \delta}.
    \end{align*}
    
    Therefore,
    \begin{align*}
        \PP(\|\nu_{i^*} - \nu_1\|_2 \geq C \delta) \leq \PP(T_1 \geq C \delta) \leq M \exp\left( - N C_{10} \delta^2\right),
    \end{align*}
    where the last inequality follows from the union bound and Theorem \ref{likelihood upper bound} (with \(M\) being the cardinality of the packing set), assuming that \(\delta > C_1\sqrt{\epsilon}\), i.e.,
    \begin{align*}
        \PP(T_1 \geq C \delta) 
        &= \PP\left( E_1 := \{\nu_1 \prec \nu_j, \|\nu_1 - \nu_j\|_2 \geq C \delta, j \in [1, M]\} \text{ is not empty}\right)\\
        &= \PP\left( \nu_1 \prec \nu_j, \|\nu_1 - \nu_j\|_2 \geq C \delta \text{ for some } j \in [M] \right)\\
        &\leq \sum_{j=1}^M \PP\left(\nu_1 \prec \nu_j, \|\nu_1 - \nu_j\|_2 \geq C \delta \right) \\
        &\leq \sum_{j=1}^M \PP\left(\nu_1 \prec \nu_j \, \bigg | \|\nu_1 - f\|_2 \leq \delta, \|\nu_1 - \nu_j\|_2 \geq C \delta \right)\\
        &\leq M \exp\left( - N C_{10} \delta^2 \right).
    \end{align*}

    By the triangle inequality, we have
    \begin{align*}
    \PP(\|\nu_{i^*} - f\|_2 \geq (C + 1)\delta) &\leq \PP(\|\nu_{i^*} - \nu_1\|_2 + \|\nu_1 - f\|_2 \geq (C + 1)\delta) \\
    &\leq \PP(\|\nu_{i^*} - \nu_1\|_2 \geq C\delta) \\
    &\leq M \exp\left( - N C_{10} \delta^2\right).
    \end{align*}
\end{proof}

\begin{lemma}\label{lemma:concentration for large delta}
    Let $\tau_J$ be defined as in Theorem \ref{upper bound}. Suppose $\tilde{J}$ satisfies \eqref{Equation:condition on J} and $\delta_{\tilde{J}}^2 := \frac{d^2}{2^{2(\tilde{J}-1)}(C+1)^2} \geq C_1^2 \epsilon$, where $C_1$ is a sufficiently large constant. Then for each such $1 \leq j \leq \tilde{J}$ we have
    $$\PP(A_j) := \PP\left(\| f - \gamma_j \|_2 > \frac{d}{2^{j-1}}\right) \leq 2\exp(-N\tau_j^2/2)\mathbbm{1}(j > 1).$$
\end{lemma}
\begin{proof}
    This proof is inspired by the proof in \cite{AkshayPrasadan2024}. Since we have $\tilde{J}$ satisfying \eqref{Equation:condition on J} and $\delta_{\tilde{J}}^2 := \frac{d^2}{2^{2(\tilde{J}-1)}(C+1)^2} \geq C_1^2 \epsilon$, where $C_1$ is a sufficiently large constant, it is clear that for any $1 \leq j \leq \tilde{J}$, these two conditions are satisfied. For $3 \leq j \leq \tilde{J}$, if $\|\gamma_{j-1} - f\|_2 \leq \frac{d}{2^{j-2}}$, then $\gamma_{j-1} = q$ for some $q \in \Lcal(j-1) \cap \B_2(f, \frac{d}{2^{j-2}})$. Recall the definition of $T_i := T(\delta, \nu_i, \Fcal')$ in Lemma \ref{concentration likelihood}. Let $\delta_j = \frac{d}{2^{j-1}(C+1)}$, we have
    \begin{align*}
        \PP&\left( \|\gamma_j - f\|_2 > \frac{d}{2^{j-1}}, \|\gamma_{j-1} - f\|_2 \leq \frac{d}{2^{j-2}}\right) \\ 
        &\leq \sum_{q \in \Lcal(j-1) \cap B(f, \frac{d}{2^{j-2}})} \PP\left(\|\gamma_j - f\|_2 > \frac{d}{2^{j-1}}, \gamma_{j-1} = q\right)\\
        &= \sum_{q \in \Lcal(j-1) \cap B(f, \frac{d}{2^{j-2}})} \PP\left(\|\argmin_{\nu \in \Ocal(q)}T(\delta_j, \nu, \Ocal(q)) - f\|_2 > (C+1)\delta_j, \gamma_{j-1} = q\right)\\
        &\leq \sum_{q \in \Lcal(j-1) \cap B(f, \frac{d}{2^{j-2}})} \PP\left(\|\argmin_{\nu \in \Ocal(q)}T(\delta_j, \nu, \Ocal(q)) - f\|_2 > (C+1)\delta_j\right)
    \end{align*}
    By using Lemma \ref{lemma:Upper bound for the cardinality2}, we have that $\Lcal(j-1) \cap \B_2(f, \frac{d}{2^{j-2}})$ has cardinality upper bounded by $\Mcal_\Fcal^{\loc}(\frac{d}{2^{j-2}}, 2c)$. Set $\Fcal' = \B_2(q, \frac{d}{2^{j-2}}) \cap \Fcal \subset \Fcal$ and using the result in Lemma \ref{lemma:Upper bound for the cardinality of the offspring} that $\Ocal(q)$ forms a $\frac{d}{2^{j-2}c} = \delta_j$-covering of $\Fcal'$ with cardinality bounded by $\Mcal_{\Fcal}^{\loc}\left(\frac{d}{2^{j-2}}, 2c\right)$. Applying Lemma \ref{concentration likelihood}, we have
    \begin{align*}
        \PP&\left( \|\gamma_j - f\|_2 > \frac{d}{2^{j-1}}, \|\gamma_{j-1} - f\|_2 \leq \frac{d}{2^{j-2}}\right) \\
        &\leq \left[ \Mcal_\Fcal^{\loc}\left(\frac{d}{2^{j-2}}, 2c\right) \right]^2\exp\left(-C_{10}N\delta_j^2\right)\\
        &= \left[ \Mcal_\Fcal^{\loc}\left(\frac{c\tau_j}{\sqrt{C_{10}}}, 2c\right) \right]^2\exp\left(-N\tau_j^2\right).
    \end{align*}
    Here we use the definition that $\tau_j := \frac{\sqrt{C_{10}}d}{2^{j-1}(C+1)} = \frac{\sqrt{C_{10}}d}{2^{j-2}c} = \sqrt{C_{10}}\delta_j$.
    Denote the event $A_j := \{\|\gamma_j - f\|_2 > \frac{d}{2^{j-1}}\}$. Then by using Lemma 3.8 in \cite{AkshayPrasadan2024}, we have for any $1 \leq j \leq \tilde{J}$,
    $$\PP(A_j) \leq \PP(A_1) + \PP(A_2 \cap A_1^c) + \sum_{j = 3}^j\PP(A_j \cap A_{j-1}^c).$$
    Note that $\PP(A_1) = 0$ and we have already bounded $\PP(A_j \cap A_{j-1}^c)$ for $3 \leq j \leq \tilde{J}$. Recall that in the algorithm for the second level, we construct a maximal $\frac{d}{c}$-packing of $B(\nu_1, d) \cap \Fcal$, which is also a $\frac{d}{c}$-covering. Thus by using Lemma \ref{concentration likelihood}, we have
\begin{align*}
        \PP(A_2 \cap A_1^c) &= \PP(A_2) \leq \Mcal_\Fcal^{\loc}(d, c)\exp\left(-C_{10}N\delta_2^2\right)\\
        &\leq \Mcal_\Fcal^{\loc}(d, c)\exp\left(-C_{10}N\delta_2^2\right)\\
        &\leq \left[\Mcal_\Fcal^{\loc}\left(\frac{c\tau_2}{\sqrt{C_{10}}}, 2c\right)\right]^2\exp\left(-C_{10}N\delta_2^2\right)
\end{align*}

Since $\tau_j$ decreases with respect to $j$ and $\Mcal_\Fcal^{\loc}$ is non-increasing with respect to $\tau_j$ by Lemma $1.4$ in \cite{AkshayPrasadan2024}, we can bound $\Mcal_\Fcal^{\loc}(\frac{c\tau_j}{\sqrt{C_{10}}}, 2c)$ with $\Mcal_\Fcal^{\loc}(\frac{c\tau_J}{\sqrt{C_{10}}}, 2c)$ for any $j \leq J$. Thus, for $1 \leq j \leq \tilde{J}$, we have
\begin{align*}
    \PP\left( \|\gamma_j - f\|_2 > \frac{d}{2^{j-1}}\right) &\leq \left[\Mcal_\Fcal^{\loc}\left(\frac{c\tau_j}{\sqrt{C_{10}}}, 2c\right)\right]^2\sum_{j = 2}^j\exp\left(-N\tau_j^2\right)\\
    & \leq \mathbbm{1}(j > 1)\left[\Mcal_\Fcal^{\loc}\left(\frac{c\tau_j}{\sqrt{C_{10}}}, 2c\right)\right]^2\frac{a_j}{1-a_j},
\end{align*}
where $a_j = \exp\left(-N\tau_j^2\right)$. Note that since $\phi < \frac{3}{2}$, we have $a_j < 1$.

Now suppose $N\tau_j^2 > 2\log\left[ \Mcal_{\Fcal}^{\loc}\bigg( \tau_j\frac{c}{\sqrt{C_{10}}}, 2c\bigg)\right]^2 \vee \log 2$, i.e., \eqref{Equation:condition on J} holds. And note that $N\tau_j^2 > \log 2$ implies $a_j < \frac{1}{2}$. Thus we have
\begin{align*}
    \PP\left( \|\gamma_j - f\|_2 > \frac{d}{2^{j-1}}\right) &\leq \mathbbm{1}(j > 1)\exp\left(-\frac{N\tau_j^2}{2}\right)\frac{a_j}{1-a_j}\\
    &\leq 2 \exp\left(-\frac{N\tau_j^2}{2}\right)\mathbbm{1}(j > 1)\\
    &= 2 \exp\left(-\frac{N\tau_j^2}{2}\right)\mathbbm{1}(j > 1).
\end{align*}

\end{proof}

\begin{lemma}\label{lemma:lemma3.10 in prasadan}
    Let $\tau_j$ be defined as in Theorem \ref{upper bound}. Suppose $\tilde{J}$ satisfies \eqref{Equation:condition on J} and also $\frac{d}{2^{\tilde{J}-1}(C+1)} \geq C_1 \sqrt{\epsilon}$ for some sufficient large constant $C_1$. Then if $\nu^*$ denotes the output after at least $\Bar{J}$ iterations, we have
    $$\Ebb_f\|\nu^*-f\|_2^2 \leq \frac{(5c+2)^2}{C_{10}}{\tau_{\tilde{J}}}^2 + \frac{4(5c + 2)^2}{C_{10}} \mathbbm{1}(\Bar{J} > 1) \frac{1}{N}\exp(-N{\tau_{\tilde{J}}}^2/2).$$
\end{lemma}
\begin{proof}
    This proof is inspired by the proof in \cite{AkshayPrasadan2024}.
    
    By Lemma \ref{lemma:concentration for large delta}, we have for $1 \leq j \leq \tilde{J}$ that
    $$\PP(A_j) = \PP\left(\| f - \gamma_j \|_2 > \frac{d}{2^{j-1}}\right) \leq 2\exp(-N\tau_j^2/2)\mathbbm{1}(j > 1),$$
    where event $A_j := \{\|\gamma_j - f\|_2 > \frac{d}{2^{j-1}}\}$.

    Let $\nu^*_{\Bar{J}} := \gamma_{\Bar{J}}$, which is the output of $\Bar{J} - 1$ steps. Define $B_j$ to be the event that $\|f - \nu^*_{\Bar{J}}\|_2 > \eta\tau_j$, where $\eta := \frac{1+3c}{2\sqrt{C_{10}}}$. Since \eqref{Equation:condition on J} holds for $\tilde{J}$, by Lemma $3.3$ in \cite{AkshayPrasadan2024}, we have for any $1 \leq j \leq \tilde{J} \leq \Bar{J}$
    $$\|\gamma_j - \nu^*_{\Bar{J}}\|_2 \leq \frac{d(2+4c)}{c2^j},$$
    where $\tilde{J} \leq \Bar{J}$ always holds by definition. Note that $A_j^c$ implies $B_j^c$, because by triangle inequality we have, if $\|\gamma_j - f\|_2 \leq \frac{d}{2^{j-1}}$:
    $$\| f - \nu_{\Bar{J}}^* \|_2 \leq \| f - \gamma_{j} \|_2 + \| \nu_{\Bar{J}}^* - \gamma_{j} \|_2 \leq \frac{d}{2^{j-1}} + \frac{d(2+4c)}{c2^{j}} = \frac{d}{2^{j-1}}\left(3+\frac{1}{c}\right) = \eta \tau_j.$$

    Thus we have
    $$\PP(B_j) = \PP(\|f - \nu^*_{\Bar{J}}\| > \eta\tau_j) \leq \PP(A_j) \leq 2\exp(-N\tau_j^2/2)\mathbbm{1}(j > 1).$$

    The above inequality also holds for $j \leq 1$ since $\frac{d}{2^{j-1}}$ will be greater than $d$ and both sides will be equals to $0$. Also we have $\mathbbm{1}(j > 1) \leq \mathbbm{1}(\Bar{J} > 1)$.

    Now for $x \in [\tau_{j}, \tau_{j-1})$, where $j \leq \tilde{J}$, which is equivalent to $x \geq \tau_{\tilde{J}}$. And note that $\tau_{j-1} = 2\tau_{j}$, we have that
    \begin{align*}
        \PP(\| f - \nu_{\Bar{J}}^* \|_2 > 2\eta x) & \leq \PP(\| f - \nu_{\Bar{J}}^* \|_2 > \eta\tau_{j-1})\\
        & \leq 2\exp(-N\tau_{j-1}^2/2)\mathbbm{1}(\Bar{J} > 1)\\
        & \leq 2\exp(-Nx^2/2)\mathbbm{1}(\Bar{J} > 1)
    \end{align*}

    Let $\nu^*$ be the output for $j$th step where $j \geq \Bar{J}$. Then by Lemma $3.3$ in \cite{AkshayPrasadan2024} we have for $x \geq \tau_{\tilde{J}}:$
    $$\|\nu^* - \nu^*_{\Bar{J}}\|_2 = \|\gamma_{j+1} - \gamma_{\Bar{J}-1}\|_2 \leq \frac{d(2+4c)}{c2^{\Bar{J}-1}} = \frac{4c+2}{3c+1}\eta\tau_{\Bar{J}} \leq \frac{4c+2}{3c+1}\eta\tau_{\tilde{J}} \leq \frac{4c+2}{3c+1}\eta x.$$

    Then by triangle inequality we have
    $$\| f - \nu^* \|_2 \leq \| f - \nu_{\Bar{J}}^* \|_2 + \| \nu_{\Bar{J}}^* - \nu^* \|_2 \leq \| f - \nu_{\Bar{J}}^* \|_2 + \frac{4c+2}{3c+1}\eta x.$$

    Then for $x \geq \tau_{\tilde{J}}$, let $\omega = \left(2 + \frac{4c+2}{3c+1}\right)\eta = \frac{5c+2}{\sqrt{C_{10}}}$, we have:
    \begin{align*}
        \PP\left(\| f - \nu^* \|_2 > \omega x \right) & \leq  \PP\left(\| f - \nu_{\Bar{J}}^* \|_2 \geq 2\eta x\right)\\
        & \leq 2\exp(-Nx^2/2)\mathbbm{1}(\Bar{J} > 1)
    \end{align*}
    
Thus we have
\begin{align*}
    \mathbb{E}\| f - \nu^* \|_2^2 &= \int_0^\infty 2x \PP(\| f - \nu^*\|_2 > x)\d x \\
    & = 2\omega^2\int_0^\infty x\PP(\| f - \nu^*\|_2 > \omega x)\d x \\
    & \leq  2\omega^2\int_0^{\tau_{\tilde{J}}} x\d x + 2\omega^2\int_{\tau_{\tilde{J}}}^\infty x\PP(\| f - \nu^*\|_2 > \omega x)\d x \\
    & \leq \omega^2{\tau_{\tilde{J}}}^2 + 4\omega^2 \mathbbm{1}(\Bar{J} > 1)\int_{\tau_{\tilde{J}}}^\infty x\exp(-Nx^2/2)\d x \\
    & =  \omega^2{\tau_{\tilde{J}}}^2 + 4\omega^2 \mathbbm{1}(\Bar{J} > 1) \frac{1}{N}\exp(-N{\tau_{\tilde{J}}}^2/2) \\
    & = \frac{(5c+2)^2}{C_{10}}{\tau_{\tilde{J}}}^2 + \frac{4(5c + 2)^2}{C_{10}} \mathbbm{1}(\Bar{J} > 1) \frac{1}{N}\exp(-N{\tau_{\tilde{J}}}^2/2).
\end{align*}

\end{proof}

\begin{theorem}
    Let $\nu^*$ be the output of the multistage sieve MLE which is run for $j$ steps where $j \geq \Bar{J}$. Here $\Bar{J}$ is defined as the maximal integer $J \in \mathbb{N}$, such that $\tau_J := \frac{\sqrt{C_{10}}d}{2^{(J-1)}(C+1)}$

satisfies
\[
N\tau_J^2 > 2\log \left[\Mcal_{\Fcal}^{\loc}\bigg( \tau_J\frac{c}{\sqrt{C_{10}}}, 2c\bigg)\right]^2 \vee \log 2,
\]
or $\Bar{J}=1$ if no such $J$ exists.

Then
\[
\mathbb{E}_f\|\nu^* - f\|_2^2 \lesssim \max\{ \tau_{\Bar{J}}^2, \epsilon \} \wedge d^2.
\]
We remind the reader that $c := 2(C+1)$ is the constant from the definition of local metric entropy, which is assumed to be sufficiently large. Here $C$ is assumed to satisfy (9), and $L(\alpha, \beta, C)$ is defined as per (10) in \cite{shrotriya2023revisiting}.
\end{theorem}

\begin{proof}
This proof is inspired by the proof in \cite{AkshayPrasadan2024} and \cite{shrotriya2023revisiting}. If $\Bar{J} = 1$, then $\tau_{\Bar{J}} \asymp d$. Then obviously we have $\mathbb{E}_f\|\nu^* - f\|_2^2 \leq d^2 \asymp \max\{ d^2, \epsilon \} \wedge d^2 \asymp \max\{ \tau_{\Bar{J}}^2, \epsilon \} \wedge d^2$. Thus we assume $\Bar{J} > 1$.

Let $\nu_{\Bar{J}}^* = \gamma_{\Bar{J}}$ be the output of the multistage sieve MLE which is run for $\Bar{J}-1$ steps. 

We assume the initial input is $\gamma_1$ and denote $\gamma_{k+1}$ the output of $k$ iterations of the algorithm. Since $N\tau_j^2$ decreasing with respect to $j$ and $\log \Mcal_{\Fcal}^{\loc}\bigg( \tau_j\frac{c}{\sqrt{C_{10}}}, 2c\bigg)$ non-decreasing with respect to $j$ according to the Lemma $9$ in \cite{shrotriya2023revisiting}, we conclude that $\Bar{J} < \infty$. Denote $A_j := \{ \|f - \gamma_j\|_2 > \frac{d}{2^{j-1}} \}$, for each integer $j \geq 1$.

    $(1)$ Suppose $\delta_J^2 = \frac{d^2}{2^{2(J-1)}(C+1)^2} \geq C_1^2 \epsilon$, where $C_1$ is a sufficiently large constant, for all $J \leq \Bar{J}$.

If one sets $\tau_J := \frac{\sqrt{C_{10}}d}{2^{(J-1)}(C+1)}$, we have that
\begin{equation}\label{thm11 condition}
    N\tau_J^2 > 2\log \left[\Mcal_{\Fcal}^{\loc}\bigg( \tau_J\frac{c}{\sqrt{C_{10}}}, 2c\bigg)\right]^2 \vee \log 2
\end{equation}

Then by Lemma \ref{lemma:lemma3.10 in prasadan} with $\tilde{J} = \Bar{J}$ and $N\tau_{\tilde{J}} \geq \log 2$, we have 

\begin{align*}
    \Ebb_f\|\nu^*-f\|_2^2 &\leq \frac{(5c+2)^2}{C_{10}}{\tau_{\Bar{J}}}^2 + \frac{4(5c + 2)^2}{C_{10}} \mathbbm{1}(\Bar{J} > 1) \frac{1}{N}\exp(-n{\tau_{\Bar{J}}}^2/2) \\
    & \leq \frac{(5c+2)^2}{C_{10}}{\tau_{\Bar{J}}}^2 + \frac{4(5c + 2)^2}{C_{10}} \mathbbm{1}(\Bar{J} > 1) \frac{1}{N}\exp(-\frac{1}{2}\log 2) \\
    & \leq \frac{(5c+2)^2}{C_{10}}{\tau_{\Bar{J}}}^2 + \frac{4(5c + 2)^2}{C_{10}} \mathbbm{1}(\Bar{J} > 1) \frac{{\tau_{\Bar{J}}}^2}{\log 2}\exp(-\frac{1}{2}\log 2)
\end{align*}

Since $d^2$ is always the bound of $\mathbb{E}_f\| f - \nu^* \|_2^2$, we have that $\mathbb{E}_f\| f - \nu^* \|_2^2 \lesssim \tau_{\Bar{J}}^2 \wedge d^2 \leq \max\{ \tau_{\Bar{J}}^2, \epsilon \} \wedge d^2$.

$(2)$ Suppose $\delta_{j_0}^2 = \frac{d^2}{2^{2(j_0-1)}(C+1)^2} < C_1^2 \epsilon$, where $C_1$ is a sufficiently large constant, for some $j_0 \in [1, \Bar{J}]$. Here $\Bar{J}$ is defined the same as the previous part. And suppose $j_0$ is the smallest one satisfies this constraint.

If $j_0 = 1 \text{ or } 2$, then $d \lesssim \sqrt{\epsilon}$. Then $d^2 \lesssim \max\{ \tau_{\Bar{J}}^2, \epsilon \} \wedge d^2$. And obviously, we have $\mathbb{E}_f\| f - \nu_{\Bar{J}}^* \|_2^2 \lesssim d^2$.
Suppose $j_0 > 2$, then for all $ j \in [1, j_0 - 1]$, we have $\delta_{j}^2 = \frac{d^2}{2^{2(j-1)}(C+1)^2} \geq C_1^2 \epsilon$. 

Now by using Lemma \ref{lemma:lemma3.10 in prasadan} with $\tilde{J} = j_0 - 1$, we have

\begin{align*}
    \Ebb_f\|\nu^*-f\|_2^2 \leq \frac{(5c+2)^2}{C_{10}}{\tau_{j_0-1}}^2 + \frac{4(5c + 2)^2}{C_{10}} \mathbbm{1}(\Bar{J} > 1) \frac{1}{N}\exp(-N{\tau_{j_0-1}}^2/2).
\end{align*}

Now, suppose $\sqrt{\epsilon} < \frac{\sqrt{\log 2}}{C_1 \sqrt{NC_{10}}} =: \frac{C_5}{\sqrt{N}}$, we have $$N\tau_{j_0}^2 = n \bigg(\frac{d \sqrt{C_{10}}}{2^{j_0 - 1}(C+1)}\bigg)^2 < N C_1^2 C_{10}\epsilon < \log 2.$$

Since $N\tau_{j}^2$ decreasing w.r.t. $j$ and $j_0 \leq \Bar{J}$, this means that $N\tau_{\Bar{J}}^2 < \log 2$. So this means $\Bar{J} = 1 = j_0$. As we talked about at the beginning, if $j_0 = 1$, then $d \lesssim \sqrt{\epsilon}$. Then $d \lesssim \max\{ \tau_{\Bar{J}}^2, \epsilon \} \wedge d^2$. And obviously, we have $\mathbb{E}\| f - \nu_{\Bar{J}}^* \|_2^2 \lesssim d^2$.

Suppose $\sqrt{\epsilon} \geq \frac{C_5}{\sqrt{N}}$, we have
$\frac{1}{N} \leq \frac{\epsilon}{C_5^2}$. And $\tau_{j_0 - 1}^2 = 4\tau_{j_0}^2 < 4C_1^2\epsilon$. Also according to the definition of $j_0$ we have $N\tau_{j_0 - 1}^2 = \frac{NC_{10}d^2}{2^{2(j_0-2)}(C+1)^2} \geq NC_{10}C_1^2\epsilon \geq \log 2$.

So
$$\mathbb{E}_f\| f - \nu^* \|_2^2 \leq \frac{(5c+2)^2}{C_{10}}4C_1^2\epsilon + \frac{4(5c + 2)^2}{C_{10}} \mathbbm{1}(\Bar{J} > 1) \frac{\epsilon}{C_5^2}\exp(-\frac{\log 2}{2}) \lesssim \epsilon < \max\{ \tau_{\Bar{J}}^2, \epsilon \}.$$

So we have $\mathbb{E}_f\| f - \nu^* \|_2^2 \lesssim \max\{ \tau_{\Bar{J}}^2, \epsilon \} \wedge d^2$.
\end{proof}

\section{Minimax Rate}
\begin{lemma}
    Define $\tau^* := \sup\{\tau : N\tau^2 \leq \log \Mcal_{\Fcal}^{\loc}(\tau, c)\}$, where $c$ in the definition of local metric entropy is a sufficiently large absolute constant. When $\epsilon < \frac{k}{N}$ with a constant $k > 7$, we have 
    $$\max\{{\tau^*}^2 \wedge d^2, \epsilon \wedge d^2\} = {\tau^*}^2 \wedge d^2$$
    up to a constant.
\end{lemma}
\begin{proof}
Let $\tau = 2\sqrt{\frac{k}{N}}$ with a constant $k > 7$. Note that $\log\Mcal_{\Fcal}^{\loc}(\tau/2, c) \geq \log\Mcal_{\Fcal}^{\loc}(\sqrt{k}, c)$, since the map $\tau \mapsto \log\Mcal_{\Fcal}^{\loc}(\tau, c)$ is non-increasing by Lemma $1.4$ in \cite{AkshayPrasadan2024}. For a sufficiently large constant \( c \), we can ensure that \(\sqrt{k}/c < d\). Note that when choosing \( c \), it only needs to be independent of the data (i.e., \( c \) is selected prior to observing the data).

Suppose $d > \tau$. We can put points in the diameter of the ball with radius $\tau/2 < d/2$ and make $\log\Mcal_{\Fcal}^{\loc}(\sqrt{k}, c) > k$ as long as $c$ big enough. Therefore, we have $\log\Mcal_{\Fcal}^{\loc}(\tau/2, c) \geq \log\Mcal_{\Fcal}^{\loc}(\sqrt{k}, c) > k = N(\tau/2)^2$. By the definition of the supremum, we have ${\tau^*}^2 \geq (\tau/2)^2 = \frac{k}{N}$.

Suppose $d \leq \tau$. We can put points in the diameter of the ball with radius $d/3$ and make $\log\Mcal_{\Fcal}^{\loc}(d/3, c) > k$ as long as $c$ big enough. And note that $N(d/3)^2 < N(\tau/2)^2 = k$. Therefore, by the definition of the supremum, we have ${\tau^*}^2 \geq (d/3)^2$.

Then, we have ${\tau^*}^2 \wedge d^2 \gtrsim \frac{k}{N} \wedge d^2$. Therefore, when $\epsilon < \frac{k}{N}$, we have $\max\{{\tau^*}^2 \wedge d^2, \epsilon \wedge d^2\} = {\tau^*}^2 \wedge d^2$ up to a constant.
\end{proof}

\begin{theorem}
    Define $\tau^* := \sup\{\tau : N\tau^2 \leq \log \Mcal_{\Fcal}^{\loc}(\tau, c)\}$, where $c$ in the definition of local metric entropy is a sufficiently large absolute constant. If Condition \ref{Condition: L_2 and TV} holds, the minimax rate is given by
\[
\max\{{\tau^*}^2 \wedge d^2, \epsilon \wedge d^2\} \text{ up to absolute constant factors.}
\]
\end{theorem}

\begin{proof}
When $\tau^* = 0$, by the definition of supremum we have $N\tau^2 > \log \Mcal_{\Fcal}^{\loc}(\tau, c)$ for any $\tau > 0$. This implies $\Mcal_{\Fcal}^{\loc}(\tau, c) = 1$ for sufficient small $\tau$, i.e. $\log2 > N\tau^2 > \log \Mcal_{\Fcal}^{\loc}(\tau, c)$. Then we can conclude that $\Fcal$ only contains one single point, i.e. $d = 0$. Our algorithm will trivially output that point
achieving the minimax rate of $0$. Thus it suffices to consider $\tau^* > 0$. Note that $\tau^* > 0$ implies $d > 0$ and this implies given a fixed constant $u$, we can make $\Mcal_{\Fcal}^{\loc}(u, c)$ greater than any fixed constant as long as $c$ is sufficient large.

Denote a constant $k > 7$.

\textbf{Case 1:} \(N{\tau^*}^2 > 4\log 2\) and $\epsilon \geq \frac{k}{N}$.

Define \(\delta^* := \frac{\tau^*}{\sqrt{4(\frac{1}{\alpha} \vee 1)}}\). In this case, we have:
\begin{align*}
    \log M^{\loc}(\delta^*, c) &\geq \lim_{\lambda \rightarrow 0} \log M_{\mathcal{F}}^{\loc}(\tau^* - \lambda, c) \\
    &\geq N{\tau^*}^2 \\
    &\geq \frac{N{\tau^*}^2}{2} + 2\log 2 \\
    &\geq \frac{2N{\delta^*}^2}{\alpha} + 2\log 2.
\end{align*}

This implies the sufficient condition for the lower bound in Lemma \ref{lowerbound in shrotriya}. Therefore, the minimax rate is lower bounded by \({\delta^*}^2\) up to a constant, and thus lower bounded by \({\tau^*}^2\) up to a constant. Since \(\epsilon \geq \frac{k}{N}\) with a constant $k > 7$ by assumption and $\xi(\epsilon) \gtrsim \epsilon$ by Condition \ref{Condition: L_2 and TV}, by Corollary \ref{corollary for new lower bound} the minimax rate is also lower bounded by \(\epsilon \wedge d^2\). Hence, the minimax rate is lower bounded by \(\max \{ {\tau^*}^2, \epsilon \wedge d^2 \}\). Moreover, as \(\max \{ {\tau^*}^2 \wedge d^2, \epsilon \wedge d^2 \} \leq \max \{ {\tau^*}^2, \epsilon \wedge d^2 \}\), we conclude that \(\max \{ {\tau^*}^2 \wedge d^2, \epsilon \wedge d^2 \}\) is a lower bound for the minimax rate in this case.

According to Theorem \ref{upper bound}, \(\max\{ \tau_{\Bar{J}}^2, \epsilon \} \wedge d^2 = \max\{ \tau_{\Bar{J}}^2 \wedge d^2, \epsilon \wedge d^2 \}\) is the upper bound for the minimax rate. We now aim to find a \(\tilde{\tau} \asymp \tau^*\) such that \(\tilde{\tau}^2 \geq \tau_{\Bar{J}}^2\), which implies
\[
\max\{ \tau_{\Bar{J}}^2 \wedge d^2, \epsilon \wedge d^2 \} \leq \max\{ \tilde{\tau}^2 \wedge d^2, \epsilon \wedge d^2 \}.
\]
Thus, we can conclude that \(\max\{ {\tau^*}^2 \wedge d^2, \epsilon \wedge d^2 \}\) is an upper bound of the minimax rate.

First, we construct \(\tilde{\tau} > 0\) that satisfies \eqref{Equation:condition on J}. Set \(\eta = \min\{1, \sqrt{8(C+1)^2/C_{10}}\}\) and choose a constant \(D\) such that \(D\eta > 1\). Then, let \(\tilde{\tau} = D\sqrt{2}\tau^*\). We then have
\begin{align*}
    N \tilde{\tau}^2 & = \eta^{-2} 2N (D\eta\tau^*)^2 \geq 2N (D\eta\tau^*)^2 > 2 \log M_{\mathcal{F}}^{\loc}(D\eta\tau^*,c) \\
    &\geq 2 \log M_{\mathcal{F}}^{\loc} \bigg( \frac{D\sqrt{8(C+1)^2}\tau^*}{\sqrt{C_{10}}}, c \bigg) = 2 \log M_{\mathcal{F}}^{\loc} \bigg( \frac{2(C+1)\tilde{\tau}}{\sqrt{C_{10}}} , c\bigg).
\end{align*}
The second inequality holds due to the definition of \(\tau^*\). By the assumption in this case, we have
\[
N \tilde{\tau}^2 = 2N D^2{\tau^*}^2 > 8D^2 \log 2 > \log 2.
\]
Thus, \(\tilde{\tau}\) satisfies \eqref{Equation:condition on J}.

Now, we know that the map \(\tau \mapsto \log M_{\mathcal{F}}^{\loc}(\tau, c)\) is non-increasing by Lemma $1.4$ in \cite{AkshayPrasadan2024}. Based on this, we define a non-decreasing map \(\phi: (0, \infty) \rightarrow \mathbb{R}\),
\[
\phi(x) = Nx^2 - 2\log M_{\mathcal{F}}^{\loc}\bigg(\frac{2(C+1)x}{\sqrt{C_{10}}}, c\bigg) \vee \log 2.
\]
We have \(\phi(\tilde{\tau}) > 0\).

Suppose some \(\tau_J\) for \(J \geq 1\) satisfies \eqref{Equation:condition on J}. Then, by the definition of \(\Bar{J}\) in Theorem \ref{upper bound}, we know that \(\tau_{\Bar{J} + 1} = \tau_{\Bar{J}}/2\) and \(\phi(\tau_{\Bar{J} + 1}) < 0 < \phi(\tilde{\tau})\). Thus, we have \(\tilde{\tau} \geq \tau_{\Bar{J}}/2\). If no \(\tau_J\) satisfies \eqref{Equation:condition on J}, then by the definition of \(\Bar{J}\) in Theorem \ref{upper bound}, we know that \(\Bar{J} = 1\). Then we have \(\phi(\tau_{\Bar{J}}) < 0 < \phi(\tilde{\tau})\), so \(\tilde{\tau} \geq \tau_{\Bar{J}}/2\). We have thus found a \(\tilde{\tau} \asymp \tau^*\) such that \(\tilde{\tau}^2 \geq \tau_{\Bar{J}}^2\).

\textbf{Case 2:} \(N{\tau^*}^2 \leq 4\log 2\) and $\epsilon \geq \frac{k}{N}$.

Since \(\tau^*\) is the supremum, we have \(\log M_{\mathcal{F}}^{\loc}(2\tau^*, c) < 4N(\tau^*)^2 \leq 16\log 2\). We now want to prove that \(d\), the diameter of the set \(\mathcal{F}\), cannot be greater than \(4\tau^*\). Suppose \(d > 4\tau^*\). Denote \(B\) as a ball with diameter \(4\tau^*\). Consider a diameter \(l\) in this ball. If we choose \(c\) large enough, we can place more than \(\exp(16\log 2)\) packing points on this line. So, if \(d\) is greater than \(4\tau^*\), we could find such a ball \(B\) such that \(\mathcal{F} \cap B\) contains at least one line segment greater than \(4\tau^*\). By the definition of \(\log M_{\mathcal{F}}^{\loc}\), this implies that \(\log M_{\mathcal{F}}^{\loc}(2\tau^*, c) > 16\log 2\), which contradicts \(\log M_{\mathcal{F}}^{\loc}(2\tau^*, c) < 16\log 2\). Thus, \(d \leq 4\tau^* \leq 8\sqrt{\frac{\log 2}{N}}\).

Let \(\tilde{\tau} = d\). First, we have \(2N\tilde{\tau}^2/\alpha + 2\log 2 < \frac{128\log 2}{\alpha} + 2\log 2\). Now, with a large enough \(c\), we have \(\log M_{\mathcal{F}}^{\loc}(\tilde{\tau}, c) > \frac{128\log 2}{\alpha} + 2\log 2\). Thus, \(\log M_{\mathcal{F}}^{\loc}(\tilde{\tau}, c) > 2N\tilde{\tau}^2/\alpha + 2\log 2\), which satisfies the sufficient condition of Lemma \ref{lowerbound in shrotriya}. So the minimax rate is lower bounded by \(d^2\) up to a constant, and thus lower bounded by \({\tau^*}^2 \wedge d^2\) up to a constant.

Since \(\epsilon \geq \frac{k}{N}\) with a constant $k > 7$ by assumption and $\xi(\epsilon) \gtrsim \epsilon$ by Condition \ref{Condition: L_2 and TV}, using Corollary \ref{corollary for new lower bound}, we conclude that \(\epsilon \wedge d^2\) up to a constant is a lower bound of the minimax rate. Thus, \(\max\{{\tau^*}^2 \wedge d^2, \epsilon \wedge d^2\}\) up to a constant is a lower bound of the minimax rate.

Clearly, \(d^2\) is an upper bound of the minimax rate. Since \(d^2 \leq 16 {\tau^*}^2\), \({\tau^*}^2 \wedge d^2\) is an upper bound up to a constant of the minimax rate. Thus, \(\max\{{\tau^*}^2 \wedge d^2, \epsilon \wedge d^2\}\) up to a constant is an upper bound of the minimax rate.

\textbf{Case 3:} $\epsilon < \frac{k}{N}$.

By Lemma \ref{lemma for minimax rate}, we have $\max\{{\tau^*}^2 \wedge d^2, \epsilon \wedge d^2\} = {\tau^*}^2 \wedge d^2$ up to a constant. Therefore, the result still holds.
\end{proof}

\end{document}